\documentclass[a4paper,leqno,10pt]{amsart}

\usepackage[T1]{fontenc}
\usepackage[utf8]{inputenc}

\usepackage{amsfonts,amsmath,amssymb,amsthm,mathtools}
\usepackage{lmodern}
\usepackage{tikz-cd}
\usepackage[english]{babel}
\usepackage{microtype}

\usepackage[normalem]{ulem}

\theoremstyle{plain}
\newtheorem{theo}{Theorem}[section]
\newtheorem*{theo*}{Theorem}
\newtheorem{prop}[theo]{Proposition}
\newtheorem{lemm}[theo]{Lemma}
\newtheorem{coro}[theo]{Corollary}
\newtheorem{conj}[theo]{Conjecture}
\theoremstyle{definition}
\newtheorem{defi}[theo]{Definition}
\newtheorem{exem}[theo]{Example}

\newtheorem{rema}[theo]{Remark}

\newcommand{\N}{\mathbf{N}}
\newcommand{\Z}{\mathbf{Z}}
\newcommand{\Q}{\mathbf{Q}}

\newcommand{\Zp}{\mathbf{Z}_p}

\newcommand{\Qp}{\mathbf{Q}_p}

\newcommand{\Fp}{\mathbf{F}_{\! p}}

\DeclareMathOperator{\gal}{Gal}
\let\temp\phi
\let\phi\varphi
\let\varphi\temp
\let\temp\epsilon
\let\epsilon\varepsilon
\let\varepsilon\epsilon
\let\hom\relax
\DeclareMathOperator{\hom}{Hom}
\let\ker\relax
\DeclareMathOperator{\ker}{Ker}
\DeclareMathOperator{\im}{Im}
\DeclareMathOperator{\coker}{Coker}
\DeclareFontFamily{U}{wncy}{}
\DeclareFontShape{U}{wncy}{m}{n}{<->wncyr10}{}
\DeclareSymbolFont{mcy}{U}{wncy}{m}{n}
\DeclareMathSymbol{\Sha}{\mathord}{mcy}{"58} 

\newcommand{\HH}{\mathcal{H}}
\newcommand{\II}{\mathcal{I}}
\newcommand{\LL}{\mathcal{L}}
\newcommand{\MM}{\mathcal{M}}
\newcommand{\OO}{\mathcal{O}}
\newcommand{\PP}{\mathcal{P}}
\newcommand{\Dcrisv}{\mathbf{D}_{\mathrm{cris},v}}

\newcommand{\A}{\mathbf{A}^+_{F_v}}
\newcommand{\Bcris}{\mathbf{B}_{\mathrm{cris}}}
\newcommand{\Brig}{\mathbf{B}^+_{\mathrm{rig},F_v}}
\newcommand{\hiw}{\operatorname{H}_{\mathrm{Iw}}^1}
\newcommand{\hur}{\operatorname{H}_{\mathrm{unr}}^1}
\DeclareMathOperator{\coh}{H}
\DeclareMathOperator{\ind}{Ind}
\DeclareMathOperator{\fil}{Fil}
\DeclareMathOperator{\frob}{Frob}
\DeclareMathOperator{\gl}{GL}
\DeclareMathOperator{\col}{Col}
\DeclareMathOperator{\sel}{Sel}
\DeclareMathOperator{\corank}{corank}
\DeclareMathOperator{\tam}{Tam}
\let\det\relax
\DeclareMathOperator{\det}{det}

\begin{document}
\author[G. Ponsinet]{Gautier Ponsinet}
\address{Max Planck Institut for Mathematics, Vivatsgasse 7, 53111 Bonn, Germany}
\email{gautier.ponsinet@mpim-bonn.mpg.de}

\title{On the structure of signed Selmer groups}
\date{\today}
\begin{abstract}
	Let $F$ be a number field unramified at an odd prime $p$ and
	$F_\infty$ be the $\Zp$-cyclotomic extension of $F$.
	Generalizing Kobayashi plus/minus Selmer groups for elliptic
	curves, B\"uy\"ukboduk and Lei have defined modified Selmer
	groups, called signed Selmer groups, for certain non-ordinary
	$\gal(\overline{F}/F)$-representations. In particular, their
	construction applies to abelian varieties defined over $F$ with
	good supersingular reduction at primes of $F$ dividing $p$.
	Assuming that these Selmer groups are cotorsion
	$\Zp[[\gal(F_\infty/F)]]$-modules, we show that they have no
	proper sub-$\Zp[[\gal(F_\infty/F)]]$-module of finite index. We
	deduce from this a number of arithmetic applications. On
	studying the Euler-Poincaré characteristic of these Selmer
	groups, we obtain an explicit formula on the size of the
	Bloch-Kato Selmer group attached to these representations.
	Furthermore, for two such representations that are isomorphic
	modulo $p$, we compare the Iwasawa-invariants of their signed
	Selmer groups.
\end{abstract}
\maketitle
\tableofcontents

\section*{Introduction}
Let $F$ be a number field and $E$ be an elliptic curve defined over $F$.
Let $p$ be an odd prime and $F_\infty$ the $\Zp$-cyclotomic extension of
$F$ (see \S \ref{subsec:cyclo}). On the algebraic side of the Iwasawa
theory for $E$ developed by Mazur~\cite{Mazur72} is the $p$-Selmer group
associated to $E$ over $F_\infty$, denoted by $\sel_p(E/F_\infty)$,
which is naturally a discrete $\Zp[[\gal(F_\infty/F)]]$-module. The
Selmer group contains arithmetic information of the curve,
\textit{e.g.} it fits in the exact sequence of groups
\[
	0 \rightarrow E(F_\infty) \otimes \Qp/\Zp \rightarrow
	\sel_p(E/F_\infty) \rightarrow \Sha_p(E/F_\infty) \rightarrow 0,
\]
where $E(F_\infty)$ is the group of $F_\infty$-rational points and
$\Sha_p(E/F_\infty)$ the $p$-primary component of the Tate-Shafarevich
group of $E$ over $F_\infty$. One goal of Iwasawa theory is to
understand the structure of $\sel_p(E/F_\infty)$ as
$\Zp[[\gal(F_\infty/F)]]$-module. 

When $E$ has good ordinary reduction at primes of $F$ dividing $p$, a
conjecture of Mazur (proved by Kato~\cite{Kato04} when $F=\Q$) states that the
Pontryagin dual of $\sel_p(E/F_\infty)$ is a torsion
$\Zp[[\gal(F_\infty/F)]]$-module. Furthemore, assuming Mazur's
conjecture and that the $E(F)$ has no $p$-torsion,
Greenberg~\cite[Proposition 4.14]{GreenbergIwasawaEll}
showed that $\sel_p(E/F_\infty)$ has no proper
sub-$\Zp[[\gal(F_\infty/F)]]$-module of finite index.

When $E$ has good supersingular reduction at some prime above $p$, the
Selmer group over $F_\infty$ is no longer a cotorsion
$\Zp[[\gal(F_\infty/F)]]$-module. In the case $F=\Q$ and $a_p=0$ (holds
whenever $p\geqslant 5$), Kobayashi~\cite{Kobayashi} has defined the so
called plus and minus (or signed) Selmer groups,  and proves that the
Pontryagin duals of these Selmer groups are torsion
$\Zp[[\gal(\Q_\infty/\Q)]]$-modules. Kim~\cite{KimSubmodules} has then extended the
definition of these Selmer groups to number fields $F$ where $p$ is
unramified and generalized Greenberg's result and showed that if the
signed Selmer groups of $E$ over $F_\infty$ are cotorsion
$\Zp[[\gal(F_\infty/F)]]$-modules, then they have no proper
submodule of finite index (for one of the signed Selmer group, namely
the plus one, he requires the additional assumption that $p$ splits
completely in $F$ and is totally ramified in $F_\infty$). This
assumption has recently been removed by Kitajima and Otsuki, see
\cite{KitajimaOtsuki18}. 

Kobayashi's construction of signed Selmer groups has been generalized to
many situations (\cite{IovitaPollack,LLZWach,KimZp2,BLIntegral}). 
In~\cite{BLIntegral}, using $p$-adic Hodge theory machinery,
B\"uy\"ukboduk and Lei have defined signed Selmer groups for certain
non-ordinary Galois representations of $\gal(\overline{F}/F)$ (see \S
\ref{subsec:motives} and \ref{subsec:Dieudonne} for hypotheses). In
particular, their construction applies to abelian varieties defined over
$F$ with good supersingular reduction at primes of $F$ dividing $p$. The
definition of the signed Selmer groups depends on a choice of a basis
for the Dieudonn\'e module associated to the representation.  For such a
basis, we may attach to each of its subset $\underline{I}$ of some
prescribed cardinality a signed Selmer group, which we denote in this
introduction by $\sel_{\underline{I}}(T/F_\infty)$, where $T$ is a
Galois representation (a free $\Zp$-module of finite rank with a
continuous action of the absolute Galois group of $F$) to which the
construction of \textit{op.  cit.} applies. They conjectured these signed
Selmer groups to be cotorsion $\Zp[[\gal(F_\infty/F)]]$-modules.  Let
$T^*$ be the Tate dual of $T$. The Dieudonn\'e module associated to
$T^*$ is the dual of the Dieudonn\'e module of $T$ and we denote by
$\underline{I^c}$ the ``subbasis'' dual to $\underline{I}$. We prove:
\begin{theo*}[{{Theorem \ref{theo:submodules}}}]
	Assume that the Pontryagin dual of both $\sel_{\underline{I}}(T/F_\infty)$ and
	$\sel_{\underline{I^c}}(T^*/F_\infty)$ are torsion
	$\Zp[[\gal(F_\infty/F)]]$-modules, then
	$\sel_{\underline{I}}(T/F_\infty)$ has no proper
	sub-$\Zp[[\gal(F_\infty/F)]]$-module of finite index.
\end{theo*}

For a good choice of basis of the Dieudonn\'e module, one can relate the
signed Selmer groups to Bloch-Kato's Selmer groups. For such a basis,
assuming that the Bloch-Kato Selmer group of $T$ over $F$ is finite, our
theorem above allows us to  employ Greenberg's strategy in
\cite[Theorem 4.1]{GreenbergIwasawaEll}  to
compute the Euler-Poincar\'e characteristic of the signed Selmer groups.
We may relate the leading term of the characteristic series of these
Selmer groups to a product of Tamagawa numbers associated to the
represenatation and the cardinal of the Bloch-Kato's Selmer group (see
Corollary~\ref{coro:EPchar}).

In the final part of the article, we study congruences of signed Selmer
groups. If $E$ and $E^\prime$ are elliptic curves defined over $\Q$ with
good ordinary reduction at $p$ and such that $E[p] \simeq E^\prime[p]$
as Galois modules, Greenberg and Vatsal~\cite{GreenbergVatsal} have
studied the consequences of such a congruences in Iwasawa theory. In
particular, assuming Mazur's conjecture, they proved that the
$\mu$-invariant of $\sel_p(E/\Q_\infty)$ vanishes if and only if that of
$\sel_p(E^\prime/\Q_\infty)$ vanishes, and that when these $\mu$-invariants do
vanish, the $\lambda$-invariants of some \emph{non-primitive Selmer
groups} associated to $E$ and $E^\prime$ over $\Q_\infty$ are equal.
Kim~\cite{KimInvariants} generalized this result to the plus and minus
Selmer groups in the supersingular case. We prove a version of this
result in the setting of \cite{BLIntegral}.
\begin{theo*}[{{Theorem \ref{theo:congruences}}}]
	Let $T$ and $T^\prime$ be Galois representations to which the
	construction of \cite{BLIntegral} applies.  Assume that $T/pT
	\simeq T^\prime/pT^\prime$ as Galois modules and that the
	Pontryagin dual of the signed Selmer groups associated to $T$,
	$T^*$, $T^\prime$ and $T^{\prime,*}$ are torsion
	$\Zp[[\gal(F_\infty/F)]]$-modules.
	Then the $\mu$-invariant of $\sel_{\underline{I}}(T/F_\infty)$
	vanishes if and only if that of
	$\sel_{\underline{I}}(T^\prime/F_\infty)$ vanishes. Furthemore, when
	these $\mu$-invariants do vanish, the $\lambda$-invariants of
	the $\underline{I}$-signed non-primitive Selmer groups
	associated to $T$ and $T^\prime$ over $F_\infty$ are equal.
\end{theo*}
The main ingredient is a result of Berger~\cite{BergerLimites} who showed that the
congruence $T/pT \simeq T^\prime/pT^\prime$ of Galois module induces a
congruence modulo $p$ on the Wach module associated to $T$ and
$T^\prime$. This allows to keep track of the congruence through
B\"uy\"ukboduk and Lei's construction.

\subsection*{Acknowledgement}
This article is part of the author's Ph.D. thesis at University Laval. The author would like
to thank his supervisor Antonio Lei for his support and patient
explanation throughout this project, as well as Kazim B\"uy\"ukboduk, 
Daniel Delbourgo and Jeff Hatley for their time and help.
The author would also like to thank B.D. Kim.
Finally, the author thank the anonymous referee for very useful comments and suggestions on an earlier version of the article, which led to many improvements.

\section{Coleman maps and signed Selmer groups}
In this section, we fix notations and recall results from
\cite{BLIntegral} that we shall need.

\subsection{Cyclotomic extension and Iwasawa algebra}
\label{subsec:cyclo}
Choose once and for all an odd prime $p$. Let $F$ be a number field 
unramified at $p$. We fix $\overline{F}$ an algebraic closure of $F$ and
denote by $G_F = \gal(\overline{F}/F)$ the absolute Galois group of $F$.
If $v$ is a prime of $F$, we denote by $F_v$ the completion of $F$ at
$v$, $\OO_{F_v}$ its ring of integers and $G_{F_v}$ the decomposition
subgroup of $v$ in $G_F$. Let $\mu_{p^n}$ be the group of $p^n$-th roots
of unity for every $n \geqslant 1$ and $\mu_{p^\infty} = \cup_{n
\geqslant 1} \mu_{p^n}$. We set $F(\mu_{p^\infty}) = \cup_{n \geqslant
1} F(\mu_{p^n})$ the $p^\infty$-cyclotomic extension of $F$ inside
$\overline{F}$. For every $n \geqslant 1$, we choose a generator
$\epsilon^{(n)}$ of $\mu_{p^n}$ with the compatibilities
$(\epsilon^{(n+1)})^p = \epsilon^{(n)}$, so that $\varprojlim_n
\epsilon^{(n)}$ is a generator of $\varprojlim_n \mu_{p^n} \simeq
\Zp(1)$. The cyclotomic character $\chi : G_F \rightarrow \Zp^*$ is
defined by the relations $g(\epsilon^{(n)}) =
(\epsilon^{(n)})^{\chi(g)}$ and it induces a isomorphism $\chi
: \gal(F(\mu_{p^\infty})/F) \simeq \Zp^*$. In particular, the group
$\gal(F(\mu_{p^\infty})/F)$ decomposes as $\Gamma \times \Delta$ with
$\Gamma \simeq \Zp$ and $\Delta \simeq \Z/(p-1)\Z$. For every $n
\geqslant 0$, we denote by $\Gamma_n$ the unique subgroup of $\Gamma$ of
index $p^n$. We set $F_\infty = F(\mu_{p^\infty})^\Delta$ and $F_n =
F_\infty^{\Gamma_n}$ for every $n \geqslant 0$.

For $n \geqslant 1$, we set $\Lambda_n =
\Zp[\gal(F(\mu_{p^n}/F)]$. Let $\Lambda =
\Zp[[\gal(F(\mu_{p^\infty})/F)]] = \varprojlim_n \Lambda_n$ be the Iwasawa algebra
of $\gal(F(\mu_{p^\infty})/F)$ over $\Zp$. The above decomposition of
$\gal(F(\mu_{p^\infty})/F)$ implies that $\Lambda =
\Zp[\Delta][[\Gamma]]$. Furthermore, we have an isomorphism
$\Zp[[\Gamma]] \simeq \Zp[[X]]$ induced by $\gamma \mapsto X+1$ where
$\gamma$ is a topological generator of $\Gamma$. For $n \geqslant 1$,
let $\omega_n(X) = (X+1)^p-1$, then this isomorphism induces $\Lambda_n
\simeq \Zp[\Delta][X]/(\omega_n)$.

For a Dirichlet character $\eta$ on $\Delta$ and a $\Lambda$-module $R$,
let $R^\eta$ be the isotypic component of $R$, which is given by
$e_\eta \cdot R$ where $e_\eta = \frac{1}{| \Delta |} \sum_{\delta \in
\Delta} \eta^{-1}(\delta)\delta$. Note that $R^\eta$ is naturally a
$\Zp[[\Gamma]]$-module. We will say that a $\Lambda$-module $R$ has rank
$r$ if $R^\eta$ has rank $r$ over $\Zp[[\Gamma]]$ for all characters
$\eta$ on $\Delta$.

Given a finitely generated torsion $\Zp[[\Gamma]]$-module $R$,
there exists a pseudo-isomorphism (\textit{i.e.} a morphism of
$\Zp[[\Gamma]]$-modules with finite kernel and cokernel)
\[
	R \rightarrow \bigoplus_{i=1}^n \Zp[[\Gamma]]/(p^{l_i}) \oplus
	\bigoplus_{j=1}^m \Zp[[\Gamma]]/(f_j^{k_j})
\]
where $f_j \in \Zp[X]$ are distinguished irreducible polynomials
(identifying $\Zp[[\Gamma]]$ and $\Zp[[X]]$). Furthermore,
the ideals $(p^{l_j})$ and $(f_j^{k_j})$ are uniquely determined by $R$ up to ordering.
The characteristic ideal of $R$ is then defined by $\prod_{i,j}
(p^{l_i})\cdot(
f_j^{k_j} ) \subset
\Zp[[\Gamma]]$. The $\mu$-invariant of $R$ is defined by $\sum_{i=1}^n
l_i$ and the $\lambda$-invariant of $R$ by $\sum_{j=1}^m k_j\cdot \deg
f_j$.

\subsection{Motives} \label{subsec:motives}
Let $\MM$ be a motive defined over $F$ with coefficients in $\Q$ in the sense of \cite{FontainePerrinRiou}.
We denote by $\MM_p$ its $p$-adic realization and we fix $T$ a
$G_F$-stable $\Zp$-lattice inside $\MM_p$.
Let $g = \dim_{\Qp}(\ind_F^\Q \MM_p)$ and $g_+ = \dim_{\Qp}(\ind_F^\Q
\MM_p)^+$ the dimension of the $+1$-eigenspace under the action of a
fixed complex conjugation on $\ind_F^\Q \MM_p$. We set $g_- = g - g_+$.
For every prime $v$ of $F$ dividing $p$, let $g_v =
\dim_{\Qp}(\ind_{F_v}^{\Qp}
\MM_p)$. We have $g = \sum_{v \mid p} g_v$.

We will assume that, for every prime $v$ of $F$ dividing $p$,
\begin{enumerate}
	\item[\textbf{(H.-T.)}] the Hodge-Tate weights of $\MM_p$,
		as a
		$G_{F_v}$-representation, are in
		$[0,1]$,
	\item[\textbf{(Cryst.)}] the $G_{F_v}$-representation $\MM_p$ is
		crystalline,
	\item[\textbf{(Tors.)}] the Galois cohomology groups
		$\coh^0(F_v,T/pT)$ and $\coh^2(F_v,T/pT)$ are trivial.
\end{enumerate}

We denote by $T^* = \hom(T,\Zp(1))$ the Tate dual of $T$ and we set 
\[
	M = T \otimes \Qp/\Zp, \quad \text{and}, \quad M^* = T^* \otimes \Qp/\Zp.
\]
We remark that the dual of $\MM$, which we denote by $\MM^*$, satisfies the hypothesis \textbf{(Cryst.)}
and \textbf{(H.-T.)}, and $T^*$, which is a $G_F$-stable $\Zp$-lattice 
inside its $p$-adic realization $\MM_p^*$, satisfies \textbf{(Tors.)}.

We also fix $\Sigma$ a finite set of primes of $F$ containing the primes
dividing $p$, the archimedean primes and the primes of ramification of
$M^*$. Let $F_\Sigma$ be the maximal extension of $F$
unramified outside $\Sigma$, so that $M^*$ is a
$\gal(F_\Sigma/F)$-module. We remark that $F(\mu_{p^\infty}) \subseteq
F_\Sigma$ since only primes above $p$ and $\infty$ can be ramified in
$F(\mu_{p^\infty})$. If $F^\prime$ is an extension of $F$ in
$F(\mu_{p^\infty})$, we will say by abuse that a prime of $F^\prime$
lies in $\Sigma$ if it divides a prime of $F$ which is in $\Sigma$.

\subsection{Dieudonn\'e modules} \label{subsec:Dieudonne}
If $v$ is a prime of $F$ dividing $p$, let $\Dcrisv(T)$ be the
Dieudonn\'e module associated to $T$ considered as a
$G_{F_v}$-representation \cite[D\'efinition V.1.1]{BergerLimites}. Then $\Dcrisv(T)$ is a free
$\OO_{F_v}$-module of rank $\dim_{\Qp} \MM_p$ equipped with a filtration
of $\OO_{F_v}$-modules $(\fil^i \Dcrisv(T))_{i \in \Z}$ such that
\[
	\fil^i \Dcrisv(T) = \left\{ \begin{array}{lr} 
		0 & \text{for } i \geqslant 1, \\
		\Dcrisv(T) & \text{for } i \leqslant -1.
	\end{array} \right.
\]
Furthermore, $\Dcrisv(\MM_p) \coloneqq \Dcrisv(T) \otimes \Qp$ is the usual
Fontaine's filtered $\phi$-module associated to $\MM_p$.

We will assume that
\begin{enumerate}
	\item[\textbf{(Fil.)}] $\sum_{v \mid p} \dim_{\Qp} \fil^0 \Dcrisv(T)
		\otimes \Qp = g_-$,
	\item[\textbf{(Slopes)}] the slopes of $\phi$ are in $]-1,0[$.
\end{enumerate}

We may choose $\{u_1,\ldots,u_{g_v}\}$ a $\Zp$-basis of $\Dcrisv(T)$
such that $\{u_1,\ldots,u_{d_v}\}$ is a basis for $\fil^0\Dcrisv(T)$ for
some $d_v$. We call such a basis a \emph{Hodge-compatible basis} and fix one for the
rest of the paper. Then, from our hypotheses, the matrix of the crystalline
Frobenius $\phi$ with respect to this basis is of the form
\begin{equation} \label{eq:matriceFrobenius}
	C_{\phi,v} = C_v \left(\begin{array}{c|c}
		I_{d_v} & 0 \\
		\hline
		0 & \frac{1}{p} I_{g_v - d_v}
	\end{array}\right)
\end{equation}
where $C_v \in \gl_{g_v}(\Zp)$ and $I_n$ is the identity matrix of size
$n$.

Let $\Dcrisv(T^*)$ be the Dieudonn\'e module associated to $T^*$. There
is a natural pairing 
\begin{equation} \label{eq:pairingDieudonne}
	\Dcrisv(T) \times \Dcrisv(T^*) \rightarrow \Dcrisv(\Zp(1))
	\simeq \Zp,
\end{equation}
with respect to which $\fil^i \Dcrisv(T^*)$ is the orthogonal complement
of $\fil^{-i}\Dcrisv(T)$ and $\phi^{-1}$ is the dual of $p\phi$.
In particular, $\Dcrisv(T^*)$ also satisfies the hypotheses
\textbf{(Fil.)} and \textbf{(Slopes)}.

\begin{exem} \label{exem:abelianvariety}
	Let $A$ be an abelian variety defined over $F$ with good
	supersingular reduction at every prime dividing $p$.
	Let $T_p(A)=\varprojlim_n A[p^n]$ be the $p$-adic Tate module of $A$ and let $V_p(A)=T_p(A)\otimes_{\Zp}\Qp$.
	Then $V_p(A)$ is a $G_F$-representation and $T_p(A)$ a $G_F$-stable $\Zp$-lattice of $V_p(A)$ which satisfy all the hypotheses
	\textbf{(Crys.)}, \textbf{(H.-T.)}, \textbf{(Tors.)},
	\textbf{(Fil.)} and \textbf{(Slopes)}.
\end{exem}

\subsection{Decomposition of Perrin-Riou's big logarithm map} \label{subsec:PRLog}
Let $v$ be a prime of $F$ dividing $p$. For $i \geqslant 0$, the
projective limit of the Galois cohomology groups $\coh^i(F_v(\mu_{p^n}),T)$
relative to the corestriction maps is denoted by
$\coh_{\mathrm{Iw}}^i(F_v,T)$. Recall that $\hiw(F_v,T)$ is a
$\Lambda$-module of rank $g_v$ \cite[Proposition A.2.3 ii)]{PRLivre}.

We set $\HH = \Qp[\Delta] \otimes_{\Qp}\HH(\Gamma)$ where $\HH(\Gamma)$
is the set of elements $f(\gamma - 1)$ with $\gamma \in \Gamma$ and
$f(X) \in \Qp[[X]]$ is convergent on the $p$-adic open unit disk.
Perrin-Riou's big logarithm map is a $\Lambda$-homomorphism \cite{PR94}
\[
	\LL_{T,v} : \hiw(F_v,T) \rightarrow \HH \otimes_{\Zp} \Dcrisv(T)
\]
which interpolates Kato's dual exponential maps \cite[II \S 1.2]{Kato91}
\[
	\exp_n^* : \coh^1(F_v(\mu_{p^n}),T) \rightarrow F_v(\mu_{p^n})
	\otimes_{\Zp} \Dcrisv(T).
\]

As in \cite{BLIntegral}, we may define for $n \geqslant 1$,
\begin{equation}
	C_{v,n} = \left( \begin{array}{c|c}
		I_{d_v} & 0 \\
		\hline
		0 & \Phi_{p^n}(1+X) I_{g_v - d_v}
	\end{array}\right) C_v^{-1} \quad \text{and} \quad
	M_{v,n} = (C_{\phi,v})^{n+1} C_{v,n} \cdots C_1,
\end{equation}
where $\Phi_{p^n}$ is the $p^n$-th cyclotomic polynomial. By
Proposition 2.5 in \emph{op. cit.}, the sequence $(M_{v,n})_{n \geqslant
1}$ converges to some $g_v \times g_v$ logarithmic matrix over $\HH$,
which we denote by $M_v$. This allows to decompose $\LL_{T,v}$ into
\begin{equation} \label{eq:decompoRegulator}
	\LL_{T,v} = (u_1, \ldots, u_{g_v}) \cdot M_v \cdot
	\begin{pmatrix}
		\col_{T,v,1}\\
		\vdots\\
		\col_{T,v,g_v}
	\end{pmatrix}
\end{equation}
where $\col_{T,v,i}, i \in \{1, \ldots,g_v\}$ are
$\Lambda$-homomorphisms from $\hiw(F_v,T)$ to $\Lambda$.
More details on the decomposition~\eqref{eq:decompoRegulator} are
given in paragraph \ref{subsec:congruencesColeman}.

\subsection{Signed Coleman maps}
Let $I_v$ be a subset of $\{1,\ldots,g_v\}$. We set 
\[
	\col_{T,I_v} : \hiw(F_v,T) \rightarrow \bigoplus_{i = 1}^{|I_v|}
	\Lambda, \quad \mathbf{z} \mapsto (\col_{T,v,i}(\mathbf{z}))_{i
	\in I_v}.
\]
These maps are called \emph{signed Coleman maps}. We recall results
about them that we shall need.
\begin{lemm}[{{\cite[Proposition 2.20, Lemma 3.22]{BLIntegral}}}] \label{lemm:BL}
	\begin{enumerate}
		\item For any character $\eta$ on $\Delta$,  the $\eta$-isotypic component of the
			image of the signed Coleman map $\im
	\col_{T,I_v}^\eta$ is a $\Zp[[\Gamma]]$-module of rank
	$|I_v|$ contained in a free $\Zp[[\Gamma]]$-module with finite index.
		\item The $\Lambda$-module $\ker \col_{T,I_v}$ is free of rank $g_v -
	|I_v|$.
	\end{enumerate}
\end{lemm}

Let 
\[
	\langle \cdot, \cdot \rangle_n : \coh^1(F_v(\mu_{p^n}),T) \times
	\coh^1(F_v(\mu_{p^n}),T^*) \rightarrow
	\coh^2(F_v(\mu_{p^n}),\Zp(1)) \simeq \Zp
\]
be Tate's local pairing. If $x = (x_n)_n$ and $y = (y_n)_n$ are elements
of $\hiw(F_v,T)$ and $\hiw(F_v,T^*)$ then the elements
\[
	\sum_{\sigma \in \gal(F_v(\mu_{p^n})/F_v)}\langle x_n,
	\sigma(y_n)
	\rangle \sigma \in \Zp[\gal(F_v(\mu_{p^n})/F_v)]
\]
are compatible under the natural projection maps
\[
\Zp[\gal(F_v(\mu_{p^{n+1}})/F_v)] \rightarrow
\Zp[\gal(F_v(\mu_{p^n})/F_v)],
\]
thus, they define an element in
$\Lambda$.
This defines Perrin-Riou's pairing 
\begin{equation} \label{eq:PRpairing}
	\hiw(F_v,T) \times \hiw(F_v,T^*) \rightarrow \Lambda.
\end{equation}

Since all our hypotheses \textbf{(Crys.)}, \textbf{(H.-T.)},
\textbf{(Tors.)}, \textbf{(Fil.)} and \textbf{(Slopes)} are satisfied by
$\MM^*$ and $T^*$, we carry out all of the constructions of
paragraph~\ref{subsec:PRLog} for $T^*$ with respect to the dual basis of
our fixed basis
$\{u_1,\ldots,u_{g_v}\}$ for the pairing~\eqref{eq:pairingDieudonne} and similarly define
signed Coleman maps for $T^*$.

Then, we have the following relation.
\begin{lemm}[{{\cite[Lemma 3.2]{LPFunctionalEq}}}]
	\label{lemm:KerOrthoComplement}
	Let $I_v$ be a subset of $\{1,\ldots,g_v\}$ and $I_v^c$ its
	complement. Then $\ker \col_{T,I_v}$ is the orthogonal
	complement of $\ker \col_{T^*,I_v^c}$ relative to Perrin-Riou's
	pairing~\eqref{eq:PRpairing}.
\end{lemm}
\begin{rema}
	In \cite{LPFunctionalEq}, there is an additional hypothesis that
	$g_+ = g_-$ and $F$ is abelian over $\Q$ with degree prime to
	$p$. However, the proof of Lemma 3.2 in \textit{op. cit.}
	applies in the setting considered in the present article in
	verbatim.
\end{rema}

\subsection{Signed Selmer groups} \label{subsec:selmer}
Let $\underline{I} = (I_v)_{v \mid p}$ be a tuple of sets indexed by the
primes $v$ of $F$ dividing $p$ and where each $I_v$ is a subset of $\{1,
\ldots,g_v\}$.

Tate's local pairing 
\[
	\coh^1(F_v(\mu_{p^n}),M^*) \times \coh^1(F_v(\mu_{p^n}),T)
	\rightarrow \coh^2(F_v(\mu_{p^n}),\mu_{p^\infty})\simeq \Qp/\Zp
\]
passes to the limit relative to restriction and corestriction and
defines a pairing
\begin{equation} \label{eq:cohompairing}
	\coh^1(F_v(\mu_{p^\infty}),M^*) \times \hiw(F_v,T) \rightarrow
	\Qp/\Zp.
\end{equation}
\begin{defi}
We define $\coh^1_{I_v}(F_v(\mu_{p^\infty}),M^*) \subseteq
\coh^1(F_v(\mu_{p^\infty}),M^*)$ as the orthogonal complement of $\ker
\col_{T,I_v}$ under the pairing~\eqref{eq:cohompairing}.
\end{defi}

The assumption $\coh^2(F_v,T/pT) = 0$ \textbf{(Tors.)} implies by Tate's
duality that $\coh^0(F_v,M^*)=0$, thus $\coh^0(F_{v,\infty},M^*) = 0$
since $\gal(F_{v,\infty}/F_v) \simeq \Zp$ is a pro-$p$-group. In
particular, by the inflation-restriction exact sequence, we have 
\[
	\coh^1(F_{v,\infty},M^*) \simeq
	\coh^1(F_v(\mu_{p^\infty}),M^*)^\Delta
\]
since the order of $\Delta$ is $p-1$ and
$\coh^0(F_v(\mu_{p^\infty}),M^*)$ is finite of order a power of $p$, and for
$n \geqslant 0$, we have
\[
	\coh^1(F_{v,n},M^*) \simeq \coh^1(F_{v,\infty},M^*)^{\Gamma_n}.
\]
We set
\[
	\coh^1_{I_v}(F_{v,\infty},M^*) =
	\coh^1_{I_v}(F_v(\mu_{p^\infty}),M^*)^\Delta \quad \text{and}
	\quad \coh^1_{I_v}(F_{v,n},M^*) =
	\coh^1_{I_v}(F_{v,\infty},M^*)^{\Gamma_n}.
\]

We also have signed Coleman maps for $T^*$. For $n \geqslant 0$, let
$(\ker \col_{T^*,I_v^c})_n$ be the image of $\ker \col_{T^*,I_v^c}$
under the natural map $\hiw(F_v,T^*) \rightarrow \coh^1(F_{v,n},T^*)$.
Again by \textbf{(Tors.)}, we have the exact sequence
\begin{equation} \label{eq:tors}
	0 \rightarrow \coh^1(F_{v,n},T^*) \xrightarrow{i_n}
	\coh^1(F_{v,n},\MM_p^*) \xrightarrow{\pi_n} \coh^1(F_{v,n},M^*)\rightarrow 0.
\end{equation}
The image of $(\ker \col_{T^*,I_v^c})_n$ 
under $i_n$ generates a $\Qp$-vector space in $\coh^1(F_{v,n},\MM_p^*)$, 
and we denote by $\overline{(\ker \col_{T^*,I_v^c})_n}$
the image of this $\Qp$-vector space in $\coh^1(F_{v,n},M^*)$ under $\pi_n$.
\begin{lemm} \label{lemm:divisible}
	For any $n \geqslant 0$, $\overline{(\ker \col_{T^*,I_v^c})_n}
	$ is the orthogonal complement of $(\ker
	\col_{T,I_v})_n$ under Tate's local pairing
\[
	\coh^1(F_{v,n},M^*) \times \coh^1(F_{v,n},T)
	\rightarrow \coh^2(F_{v,n},\mu_{p^\infty})\simeq \Qp/\Zp.
\]
	Moreover, $\overline{(\ker \col_{T^*,I_v^c})_n}
	= \coh^1_{I_v}(F_{v,n},M^*)$.
	In particular, $\coh^1_{I_v}(F_{v,n},M^*)$ is a divisible group.
\end{lemm}
\begin{proof}
	By Lemma~\ref{lemm:KerOrthoComplement} and bilinearity of Tate's
	pairing, the orthogonal
	complement of $(\ker \col_{T,I_v})_n$ under Tate's pairing
	contains $\overline{(\ker \col_{T^*,I_v^c})_n}$. The reverse
	inclusion follows from the exactness of the sequence~\eqref{eq:tors}.
	As already remarked, by \textbf{(Tors.)}, one has
	$\coh^1(F_{v,n},M^*) =
	\coh^1(F_v(\mu_{p^\infty}),M^*)^{\Gamma_n\times \Delta}$ and by
	duality $\hiw(F_v,T)_{\Gamma_n\times \Delta} =
	\coh^1(F_{v,n},T)$. It follows that $\overline{(\ker
	\col_{T^*,I_v^c})_n} = \coh^1_{I_v}(F_{v,n},M^*)$.
\end{proof}

Let $w$ be a prime of $F$ not dividing $p$ and let $K$ be a finite extension of $F_w$.
Define
\[
	\hur(K,\MM^*_p) = \ker (\coh^1(K,\MM^*_p) \rightarrow
	\coh^1(K_\mathrm{unr},\MM^*_p))
\]
where $K_\mathrm{unr}$ the maximal unramified extension of
$K$.
Let $\hur(K,M^*)$ be the image of $\hur(K,\MM^*_p)$ under the natural map
\[
	\coh^1(K,\MM^*_p) \rightarrow
	\coh^1(K,M^*)
\]
and $\hur(K,T^*)$ the inverse image of
$\hur(K,\MM^*_p)$ under
\[
	\coh^1(K,T^*) \rightarrow
	\coh^1(K,\MM_p^*).
\]
We remark that $\hur(K,M^*)$ is divisible by definition and recall
that it is the orthogonal complement of $\hur(K,T)$ under
Tate's local pairing (see \cite[Proposition 3.8]{BK}).
If $K^\prime$ is an infinite algebraic extension of $F_w$, we define the subgroup
\[
	\hur(K^\prime,M^*)=\varinjlim_{K} \hur(K,M^*) \subset \coh^1(K^\prime,M^*)
\]
where the limit runs through the finite extensions $K$ of $F_w$ contained in $K^\prime$ and is taken with respect to the restriction maps.

Let $F^\prime$ be one of $F(\mu_{p^\infty})$, $F_\infty$ or $F_n$ for some $n\geqslant 0$.
We set
\[
	\PP_{\Sigma,\underline{I}}(M^*/F^\prime) = \prod_{w \in \Sigma,
	w \nmid p} \frac{\coh^1(F_w^\prime,M^*)}{\hur(F_w^\prime,M^*)}
	\times \prod_{w \mid p}
	\frac{\coh^1(F_w^\prime,M^*)}{\coh^1_{I_v}(F_w^\prime,M^*)}.
\]

\begin{defi}
	Let $F^\prime$ be $F(\mu_{p^\infty})$, $F_\infty$, or $F_n$ for
	some $n \geqslant 0$. The $\underline{I}$-Selmer group of $M^*$
	over $F^\prime$ is defined by 
	\[
		\sel_{\underline{I}}(M^*/F^\prime) = \ker
		(\coh^1(F_\Sigma/F^\prime,M^*) \rightarrow
		\PP_{\Sigma,\underline{I}}(M^*/F^\prime))
	\]
	where the map is the composition of localization at each $w\in
	\Sigma$ followed by the projection in the appropriate quotient.
\end{defi}

Let $\II_p$ be the set of tuples $\underline{I} = (I_v)_{v \mid p}$ indexed by the
primes $v$ of $F$ dividing $p$ and where each $I_v$ is a subset of
$\{1,\ldots,g_v\}$ and such that $\sum_{v \mid p} |I_v| = g_-$.
From observations about the expected $\Lambda$-corank of the Selmer group of a supersingular abelian variety, B\"uy\"ukboduk and Lei have made the following conjecture \cite[Remark 3.27]{BLIntegral}.
\begin{conj} \label{conj:cotorsion}
	For any $\underline{I} \in \II_p$ and any even Dirichlet
	character $\eta$ on $\Delta$, 
	$\sel_{\underline{I}}(M^*/F(\mu_{p^\infty}))^\eta$ is a cotorsion
	$\Zp[[\Gamma]]$-module (\textit{i.e.} its Pontryagin dual is a
	torsion $\Zp[[\Gamma]]$-module).
\end{conj}

\begin{rema} When $F=\Q$ and $\MM$ is the Tate module of a supersingular elliptic
curve with $a_p=0$, for a good choice of basis of the Dieudonn\'e module, the signed
Selmer groups with $\underline{I} \in \II_p$ coincide with Kobayashi plus and minus Selmer groups
	\cite{Kobayashi} (see \cite[Appendix 4]{BLIntegral}).
	Conjecture~\ref{conj:cotorsion} is known in that case \textit{op. cit.}.
	Furthermore, Sprung~\cite{Sprung} as well as Lei, Loeffler and Zerbes~\cite{LLZWach} have proved that this conjecture holds in cases of $p$-supersingular elliptic curves with $a_p\neq 0$ and $p$-non-ordinary eigenforms respectively.
\end{rema}

\begin{rema}
The definition of the signed Selmer groups does not depend on the choice
of $\Sigma$.  If $\eta$ is the trivial character on $\Delta$,
then $\sel_{\underline{I}}(M^*/F(\mu_{p^\infty}))^\eta \simeq
\sel_{\underline{I}}(M^*/F_\infty)$.  It follows from the
definition that, for any $\underline{I}$, the Pontryagin dual of
$\sel_{\underline{I}}(M^*/F_\infty)$ is a finitely generated
$\Zp[[\Gamma]]$-module  since $\coh^1(F_\Sigma/F_\infty,M^*)$ is
\cite[Proposition 3]{GreenbergIwasawaRep}.  In the remainder of
this article, we study these Selmer groups.
\end{rema}

In the next section, we shall need \emph{twisted} signed Selmer groups.
Let us explain now what they are. For $s \in \Z$, we set $M^*_s = M^*
\otimes \chi^s_{| \Gamma}$ where $\chi_{|\Gamma} : \Gamma \simeq \Zp$.
As a $\gal(\overline{F}/F_\infty)$-module, $M_s^* = M^*$, thus
$\coh^1(F_\infty,M^*_s) = \coh^1(F_\infty,M^*)\otimes \chi^s_{| \Gamma}$ and for a
prime $v$ of $F$, $\coh^1(F_{v,\infty},M^*_s) =
\coh^1(F_{v,\infty},M^*)\otimes \chi^s_{| \Gamma}$ and
$\coh^0(F_{v,\infty},M^*_s)=0$. At primes dividing $p$, we set
\[
	\coh^1_{I_v}(F_{v,\infty},M^*_s) =
	\coh^1_{I_v}(F_{v,\infty},M^*) \otimes \chi^s_{| \Gamma}.
\]
Therefore, for $F^\prime$ being $F_\infty$ or $F_n$ for some $n
\geqslant 0$, we can define twisted $\underline{I}$-Selmer groups
$\sel_{\underline{I}}(M^*_s/F^\prime)$ as above with local condition at
$p$ induced by $\coh_{I_v}^1(F_{v,\infty},M^*_s)$. We remark that
$\sel_{\underline{I}}(M^*_s/F_\infty) \simeq
\sel_{\underline{I}}(M^*/F_\infty) \otimes \chi^s_{| \Gamma}$ as
$\Zp[[\Gamma]]$-modules. 

Similarly, we can define signed Selmer groups for $M$
using the signed Coleman maps $\col_{T^*,I_v}$, as well as twisted
signed Selmer groups for $M$ as above.
We remark that if $\underline{I}$ is an element of $\II_p$, then
$\underline{I}^c = (I^c_v)_{v \mid p}$ satisfies $\sum_{v \mid p} I_v^c
= g - g_- = g_+ = \dim_{\Qp} (\ind_F^\Q \MM_p^*)^-$. In particular,
Conjecture~\ref{conj:cotorsion} is expected to hold for the signed Selmer
groups of $M$.

\subsection{Bloch-Kato's Selmer groups} \label{subsec:BK}
Let $n \geqslant 0$ and $w$ be a prime of $F_n$ dividing $p$. 
We recall that Bloch and Kato~\cite{BK} defined 
the $\Qp$-subspace of $\coh^1(F_{n,w},\MM_p^*)$ 
\[
	\coh^1_f(F_{n,w},\MM_p^*) =
	\ker(\coh^1(F_{n,w},\MM_p^*) \rightarrow
	\coh^1(F_{n,w},\Bcris \otimes \MM_p^*))
\]
where $\Bcris$ is Fontaine's ring of crystalline periods~\cite{Fontaine94}.
Let $\coh^1_f(F_{n,w},M^*)$ be the image of
$\coh^1_f(F_{n,w},\MM_p^*)$ under the natural map
\[
	\coh^1(F_{n,w},\MM_p^*) \rightarrow
	\coh^1(F_{n,w},M^*).
\]
We set 
\[
	\PP_{\Sigma,f}(M^*/F_n) = \prod_{w \in \Sigma,
	w \nmid p} \frac{\coh^1(F_{n,w},M^*)}{\hur(F_{n,w},M^*)}
	\times \prod_{w \mid p}
	\frac{\coh^1(F_{n,w},M^*)}{\coh^1_f(F_{n,w},M^*)}.
\]
Then, the Bloch-Kato's Selmer group
of $M^*$ over $F_n$ is defined by
\[
	\sel_{\mathrm{BK}}(M^*/F_n) = \ker(\coh^1(F_\Sigma/F_n,M^*)
	\rightarrow \PP_{\Sigma,f}(M^*/F_n))
\]
and we set $\sel_{\mathrm{BK}}(M^*/F_\infty) = \varinjlim_n
\sel_{\mathrm{BK}}(M^*/F_n)$.

Recall that the definition of the signed Coleman
maps and thus of the signed Selmer groups depends on a choice of Hodge-compatible basis
of $\oplus_{v \mid p} \Dcrisv(T)$.
\begin{lemm}[{{\cite[Lemma 8.1]{BL2015}}}] \label{lemm:BKSigned}
	There exists a Hodge-compatible basis of $\oplus_{v \mid p} \Dcrisv(T)$ such
	that for any $\underline{I} \in \II_p$
	\[
		\coh^1_f(F_v,M^*) = \coh^1_{I_v}(F_v,M^*).
	\]
In particular, for such a basis,
	\[
		\sel_{\mathrm{BK}}(M^*/F) = \sel_{\underline{I}}(M^*/F).
	\]
\end{lemm}
The basis of the lemma is a \emph{strongly admissible basis} in the
sense of \cite[Definition 3.2]{BLIntegral}.

\section{Submodules of finite index}
We keep the notation of the previous section. Let $\underline{I} =
(I_v)_{v \mid p} \in \II_p$ and set $\underline{I}^c =
(I_v^c)_{v \mid p}$. The main goal of this section is to prove the following
theorem.
\begin{theo} \label{theo:submodules}
	Assume that
	$\sel_{\underline{I}}(M^*/F_\infty)$ and
	$\sel_{\underline{I}^c}(M/F_\infty)$ are cotorsion
	$\Zp[[\Gamma]]$-modules. Then
	$\sel_{\underline{I}}(M^*/F_\infty)$ has no proper
	sub-$\Zp[[\Gamma]]$-modules of finite index.
\end{theo}

\begin{rema}
 	Under the additional hypothesis
	that $F$ is abelian over $\Q$ with degree prime to $p$ and that
	$g_+ = g_-$, an algebraic
	functional equation relating
	$\sel_{\underline{I}}(M^*/F_\infty)$ and
	$\sel_{\underline{I}^c}(M/F_\infty)$ has been proved in \cite{LPFunctionalEq}. In this
	situation, if one of these $\Zp[[\Gamma]]$-modules is a
	cotorsion $\Zp[[\Gamma]]$-module, then
	they both are.
\end{rema}

\subsection{The proof of Theorem~\ref{theo:submodules}}
We begin with a ``control theorem'' for these signed Selmer groups.
\begin{lemm} \label{lemm:control}
	For all but finitely many $s \in \Z$, the kernel and cokernel
	of the restriction map
	\[
		\sel_{\underline{I}^c}(M_s/F_n) \rightarrow
		\sel_{\underline{I}^c}(M_s/F_\infty)^{\Gamma_n}
	\]
	are finite of bounded orders as $n$ varies.
\end{lemm}
\begin{proof}
	The diagram
	\begin{equation} \label{diag:control}
		\begin{tikzcd}
			0 \arrow{r} & \sel_{\underline{I}^c}(M_s/F_n)
			\arrow{r} \arrow{d} & \coh^1(F_\Sigma/F_n,M_s)
			\arrow{r} \arrow{d} &
			\PP_{\Sigma,\underline{I}^c}(M_s/F_n) \arrow{d} \\
			0 \arrow{r}  &
			\sel_{\underline{I}^c}(M_s/F_\infty)^{\Gamma_n}
			\arrow{r} &
			\coh^1(F_\Sigma/F_\infty,M_s)^{\Gamma_n} \arrow{r} &
			\PP_{\Sigma,\underline{I}^c}(M_s/F_\infty)^{\Gamma_n}
		\end{tikzcd}
	\end{equation}
	is commutative.

	By \textbf{(Tors.)}, $\coh^0(F_{v,\infty},M_s)=0$ where $v$ is
	any prime of $F$ dividing $p$, thus the central map is an
	isomorphism by the inflation-restriction exact sequence.

	We now study the kernel of the rightmost vertical map.
	For a prime $v$ of $F$ dividing $p$, the diagram
	\begin{equation} \label{diag:controlp}
		\begin{tikzcd}
			0 \arrow{r} & \coh^1_{I_v}(F_{v,n},M_s)
			\arrow{r} \arrow{d} & \coh^1(F_{v,n},M_s)
			\arrow{r} \arrow{d} &
			\frac{\coh^1(F_{v,n},M_s)}{\coh^1_{I_v}(F_{v,n},M_s)}
			\arrow{r} \arrow{d} & 0 \\
			0 \arrow{r} &
			\coh^1_{I_v}(F_{v,\infty},M_s)^{\Gamma_n}
			\arrow{r} & \coh^1(F_{v,\infty},M_s)^{\Gamma_n} \arrow{r} &
			\left(\frac{\coh^1(F_{v,\infty},M_s)}{\coh^1_{I_v}(F_{v,\infty},M_s)}\right)^{\Gamma_n}
		\end{tikzcd}
	\end{equation}
	is commutative. The central vertical map is an isomorphism by
	the inflation-restriction exact sequence and the left-most
	vertical one is an isomorphism by definition, thus it follows
	from the snake lemma applied to the diagram~\eqref{diag:controlp} that the map 
	\[
		\frac{\coh^1(F_{v,n},M_s)}{\coh^1_{I_v}(F_{v,n},M_s)}
		\rightarrow \left(\frac{\coh^1(F_{v,\infty},M_s)}{\coh^1_{I_v}(F_{v,\infty},M_s)}\right)^{\Gamma_n}
	\]
	is an injection.

	For a prime $w$ of $F_n$ not dividing $p$ and a prime $w^\prime$ of
	$F_\infty$ above $w$, the diagram 
	\begin{equation} \label{diag:controlunr}
		\begin{tikzcd}
			0 \arrow{r} & \hur(F_{n,w},M_s)
			\arrow{r} \arrow{d} & \coh^1(F_{n,w},M_s)
			\arrow{r} \arrow{d} &
			\frac{\coh^1(F_{n,w},M_s)}{\hur(F_{n,w},M_s)}
			\arrow{r} \arrow{d} & 0 \\
			0 \arrow{r} &
			\hur(F_{\infty,w^\prime},M_s)^{\Gamma_n}
			\arrow{r} & \coh^1(F_{\infty,w^\prime},M_s)^{\Gamma_n} \arrow{r} &
			\left(\frac{\coh^1(F_{\infty,w^\prime},M_s)}{\hur(F_{\infty,w^\prime},M_s)}\right)^{\Gamma_n}
		\end{tikzcd}
	\end{equation}
	is commutative.
	If $w$ is archimedean, since $p$ is odd, then
	$\coh^1(F_{\infty,w^\prime},M_s)$ is trivial, and if $w$ is
	non-archimedean, then $\hur(F_{\infty,w^\prime},M_s)$ is trivial
	\cite[\S A.2.4]{PRLivre}.

	We now look at the kernel of the central vertical map in diagram \eqref{diag:controlunr}. From the
	inflation-restriction exact sequence, it is
	$\coh^1(F_{\infty,w^\prime}/F_{n,w},M_s^{G_{F_{\infty,w^\prime}}})$.
	If $w$ is archimedean, it splits completely in
	$F_\infty/F_n$ so this group is trivial.
	If $w$ is non-archimedean, it finitely decomposes in
	$F_\infty/F_n$, so that 
	$\gal(F_{\infty,w^\prime}/F_{n,w}) \simeq \Zp$ and is topologically
	generated by an element $\gamma_n$. Thus
	$\coh^1(F_{\infty,w^\prime}/F_{n,w},M_s^{G_{F_{\infty,w^\prime}}})$
	is isomorphic to 
	\[
		M_s^{G_{F_{\infty,w^\prime}}} / (\gamma_n - 1)
		M_s^{G_{F_{\infty,w^\prime}}}.
	\]
	One has the short exact sequence 
	\begin{center}
		\begin{tikzcd}
			0 \arrow{r} & M_s^{G_{F_{n,w}}} \arrow{r} &
			M_s^{G_{F_{\infty,w^\prime}}}\arrow{r}{(\gamma_n
			- 1)} &
			M_s^{G_{F_{\infty,w^\prime}}} \arrow{r} & 
		M_s^{G_{F_{\infty,w^\prime}}} / (\gamma - 1)
			M_s^{G_{F_{\infty,w^\prime}}} \arrow{r} & 0.
		\end{tikzcd}
	\end{center}	
	For all but finitely many $s \in \Z$, $M_s^{G_{F_{n,w}}}$ is finite for every
	$n$, hence $M_s^{G_{F_{\infty,w^\prime}}} / (\gamma - 1)
	M_s^{G_{F_{\infty,w^\prime}}}$ is finite. So,
	$(M_s^{G_{F_{\infty,w^\prime}}})_\mathrm{div}$ the maximal
	divisible subgroup of $M_s^{G_{F_{\infty,w^\prime}}}$ is contained in
	$(\gamma_n - 1)M_s^{G_{F_{\infty,w^\prime}}}$ and the order of
	$M_s^{G_{F_{\infty,w^\prime}}} / (\gamma_n - 1)
	M_s^{G_{F_{\infty,w^\prime}}}$ is bounded by the one of
	$M_s^{G_{F_{\infty,w^\prime}}}/(M_s^{G_{F_{\infty,w^\prime}}})_\mathrm{div}$. 

	Thus, the snake lemma applied to the diagram~\eqref{diag:controlunr} implies that the map
	\[
	\frac{\coh^1(F_{n,w},M_s)}{\hur(F_{n,w},M_s)} \rightarrow \left(\frac{\coh^1(F_{\infty,w^\prime},M_s)}{\hur(F_{\infty,w^\prime},M_s)}\right)^{\Gamma_n}
	\]
	has finite kernel of bounded orders as $n$ varies.

	Finally, the result follows from the snake lemma applied to the
	diagram~\eqref{diag:control}.
\end{proof}

\begin{prop} \label{prop:surjection}
	Assume that
	$\sel_{\underline{I}^c}(M/F_\infty)$ is a cotorsion
	$\Zp[[\Gamma]]$-module. Then for all but finitely many $s \in
	\Z$, the map
	\[
		\coh^1(F_\Sigma/F,M^*_s) \rightarrow
		\PP_{\Sigma,\underline{I}}(M_s^*/F)
	\]
	is surjective,
	and, for all $s \in \Z$, the map
	\[
		\coh^1(F_\Sigma/F_\infty,M^*_s) \rightarrow
		\PP_{\Sigma,\underline{I}}(M^*_s/F_\infty)
	\]
	is surjective.
\end{prop}
\begin{proof}
	If $\sel_{\underline{I}^c}(M/F_\infty)$ is a cotorsion
	$\Zp[[\Gamma]]$-module, then, for all but finitely many $s \in
	\Z$, $(\sel_{\underline{I}^c}(M/F_\infty) \otimes
	\chi^s)^{\Gamma_n} =
	(\sel_{\underline{I}^c}(M_s/F_\infty))^{\Gamma_n}$ is finite
	for every $n$. Thus, by Lemma~\ref{lemm:control} and possibly
	avoiding another finite number of $s \in \Z$,
	$\sel_{\underline{I}^c}(M_s/F_n)$ is finite for every $n$. For
	such an $s$ and any $n$, the finiteness of
	$\sel_{\underline{I}^c}(M_s/F_n)$ and Lemma~\ref{lemm:divisible}
	 allow us to apply \cite[Proposition
	4.13]{GreenbergIwasawaEll} which
	says that the cokernel of 
	\[
		f_{n,-s}:\coh^1(F_\Sigma/F_n,M_{-s}^*) \rightarrow
		\PP_{\Sigma,\underline{I}}(M_{-s}^*/F_n)
	\]
	is the Pontryagin dual of $\coh^0(F_n,M_s)$. By
	\textbf{(Tors.)}, $\coh^0(F,M) = 0$, thus $\coh^0(F_\infty,M)=0$
	as $\Zp$ is a pro-$p$-group. Furthermore $M_s \simeq M$ as
	$\gal(\overline{F}/F_\infty)$-modules, hence
	$\coh^0(F_\infty,M_s) = 0$ and finally $\coh^0(F_n,M_s)$ is
	trivial for any $n$. 
	Therefore, the map $f_{n,-s}$ is surjective for any $n$.
	Passing to direct limit relative to restriction maps, the surjection of the maps $f_{n,-s}$ implies the surjection of 
	\[
		f_{\infty,-s}: \coh^1(F_\Sigma/F_\infty,M^*_{-s}) \rightarrow
		\PP_{\Sigma,\underline{I}}(M^*_{-s}/F_\infty).
	\]
	Since the map $f_{\infty,-s}$ is the map
	\[
		f_\infty :\coh^1(F_\Sigma/F_\infty,M^*) \rightarrow
		\PP_{\Sigma,\underline{I}}(M^*/F_\infty)
	\]
	twisted by $\chi^{-s}_{| \Gamma}$,
	this concludes the proof of the proposition.
\end{proof}

\begin{lemm} \label{lemm:surjection}
	For all $s\in \Z$, the restriction map
	\[
		\PP_{\Sigma,\underline{I}}(M^*_s/F) \rightarrow
		\PP_{\Sigma,\underline{I}}(M^*_s/F_\infty)^\Gamma
	\]
	is surjective.
\end{lemm}
\begin{proof}
	We have
	\[
		\PP_{\Sigma,\underline{I}}(M^*_s/F_\infty) = \prod_{w \in \Sigma,
	w \nmid p}
	\frac{\coh^1(F_{\infty,w},M^*_s)}{\hur(F_{\infty,w},M^*_s)}
	\times \prod_{v \mid p}
	\frac{\coh^1(F_{v,\infty},M^*_s)}{\coh^1_{I_v}(F_{v,\infty},M^*_s)} .
	\]
	If $w$ is archimedean, since $p$ is odd,
	$\coh^1(F_{\infty,w},M^*_s)$ is trivial.
	If $v$ is a non-archimedean prime of $F$ not dividing $p$, the
	surjection 
	\[
		\frac{\coh^1(F_v,M^*_s)}{\hur(F_v,M^*_s)} \rightarrow
		\left(\prod_{w\mid v}
		\frac{\coh^1(F_{\infty,w},M^*_s)}{\hur(F_{\infty,w},M^*_s)}\right)^{\Gamma}
	\]
	follows from the fact that $\hur(F_{\infty,w},M^*_s)$ is
	trivial and $\Gamma$ has $p$-cohomological dimension $1$.
	Finally, if $v$ is a prime of $F$ dividing $p$, then the Pontryagin dual
	of $\coh^1_{I_v}(F_{v,\infty},M^*)$ is contained in a free
	$\Zp[[\Gamma]]$-module by Lemma~\ref{lemm:BL}, thus
	$\coh^1_{I_v}(F_{v,\infty},M^*)_{\Gamma} = 0$. Hence, $\coh^1_{I_v}(F_{v,\infty},M^*_s)_{\Gamma} = 0$ and we have an
	exact sequence 
	\begin{equation} \label{eq:surj}
		0 \rightarrow \coh^1_{I_v}(F_{v,\infty},M^*_s)^{\Gamma}
		\rightarrow \coh^1(F_{v,\infty},M^*_s)^{\Gamma}
		\rightarrow
		\left(\frac{\coh^1(F_{v,\infty},M^*_s)}{\coh^1_{I_v}(F_{v,\infty},M^*_s)}\right)^{\Gamma}
		\rightarrow 0.
	\end{equation}
	By \textbf{(Tors.)}, we know that $\coh^1(F_v,M^*_s)
	\simeq \coh^1(F_{v,\infty},M^*_s)^\Gamma$. Thus, by definition of
	$\coh^1_{I_v}(F_v,M^*_s)$ and the exact sequence~\eqref{eq:surj},
	the map
	\[
		\frac{\coh^1(F_{v},M^*_s)}{\coh^1_{I_v}(F_{v},M^*_s)}\rightarrow
		\left(\frac{\coh^1(F_{v,\infty},M^*_s)}{\coh^1_{I_v}(F_{v,\infty},M^*_s)}\right)^{\Gamma}
	\]
	is surjective.
\end{proof}

\begin{lemm} \label{lemm:corank}
	For all $s \in \Z$, the $\Zp[[\Gamma]]$-corank of
	$\PP_{\Sigma,\underline{I}}(M^*_s/F_\infty)$ is $g_+$.  
\end{lemm}
\begin{proof}
	If $w$ is archimedean, since $p$ is odd,
	$\coh^1(F_{\infty,w},M_s^*)$ is trivial. If $w$ is a
	non-archimedean prime not dividing $p$ above a prime $v$ of $F$,
	by \cite[Proposition 2]{GreenbergIwasawaRep}, $\coh^1(F_{\infty,w},M_s^*)$ is
	cotorsion. 
	Finally, by definition, the Pontryagin dual of 
	$\frac{\coh^1(F_{v,\infty},M^*)}{\coh^1_{I_v}(F_{v,\infty},M^*)}$ 
	is isomorphic to $\ker \col_{T,I_v}$ which is of rank $g_v -
	|I_v|$ by Lemma~\ref{lemm:BL}. Since
	\[
		\frac{\coh^1(F_{v,\infty},M^*_s)}{\coh^1_{I_v}(F_{v,\infty},M^*_s)} =
	\frac{\coh^1(F_{v,\infty},M^*)\otimes\chi^s_{| \Gamma}}{\coh^1_{I_v}(F_{v,\infty},M^*)\otimes\chi^s_{| \Gamma}},
	\]
	the corank of $\frac{\coh^1(F_{v,\infty},M^*_s)}{\coh^1_{I_v}(F_{v,\infty},M^*_s)}$ is also $g_v -|I_v|$.
	Therefore, from our choice of
	$\underline{I}$, the corank of
	$\PP_{\Sigma,\underline{I}}(M^*_s/F_\infty)$ is 
	\[
		\sum_{v \mid p} g_v - |I_v| = g - g_- = g_+.
	\]
\end{proof}

Proposition~\ref{prop:surjection} and Lemma~\ref{lemm:corank} enable to compute the corank of the Bloch-Kato Selmer
group.
\begin{coro} \label{cong:coroBKrank}
	Assume that $\sel_{\underline{I}}(M^*/F_\infty)$ and
	$\sel_{\underline{I}^c}(M/F_\infty)$ are cotorsion
	$\Zp[[\Gamma]]$-modules. Then the $\Zp[[\Gamma]]$-corank of
	$\sel_{\mathrm{BK}}(M^*/F_\infty)$ is $g_+$.
\end{coro}
\begin{proof}
	By our hypotheses, the representation $\MM_p$ does not contain a
	non-trivial sub-$G_F$-representation with stricly positive Hodge-Tate
	weights at each prime $v$ of $F$ dividing $p$ and $(\MM_p)^{G_{F_{v,\infty}}}=0$. Hence, by
	\cite[Corollaire 0.5]{PRNormes}, we have 
	\begin{equation}\label{eq:PRNormes}	
	\coh^1_f(F_{v,\infty},M^*) = \coh^1(F_{v,\infty},M^*). 
	\end{equation}
	
	Thus, equation~\eqref{eq:PRNormes} combined with Proposition~\ref{prop:surjection} gives the commutative diagram
	\begin{center}
		\begin{tikzcd}	
			0 \arrow{r} & \sel_{I}(M^*/F_\infty)
			\arrow{r} \arrow{d} & \coh^1(F_\Sigma/F_\infty,M^*)
			\arrow{r} \arrow{d}{=} &
			\PP_{\Sigma,\underline{I}}(M^*/F_\infty)
			\arrow{r} \arrow{d} & 0 \\
			0 \arrow{r} & \sel_{\mathrm{BK}}(M^*/F_\infty)
			\arrow{r} & \coh^1(F_\Sigma/F_\infty,M^*)
			\arrow{r} &
			\prod_{w \in \Sigma,w\nmid p}
			\frac{\coh^1(F_{\infty,w},M^*)}{\coh^1_\mathrm{unr}(F_{\infty,w},M^*)},
		\end{tikzcd}
	\end{center}
	which induces, by the snake Lemma, the short exact sequence
	\[
		0 \rightarrow \sel_{\underline{I}}(M^*/F_\infty) \rightarrow
		\sel_{\mathrm{BK}}(M^*/F_\infty) \rightarrow
		\prod_{v \mid p}
		\frac{\coh^1(F_{v,\infty},M^*)}{\coh^1_{I_v}(F_{v,\infty},M^*)}
		\rightarrow 0.
	\]
	The Corollary follows from the hypothesis that
	$\sel_{\underline{I}}(M^*/F_\infty)$ is cotorsion and (the
	proof of) Lemma~\ref{lemm:corank}.
\end{proof}

\begin{prop} \label{prop:submodule}
	Assume that
	$\sel_{\underline{I}}(M^*/F_\infty)$ and
	$\sel_{\underline{I}^c}(M/F_\infty)$ are cotorsion
	$\Zp[[\Gamma]]$-modules. Then, for all $s\in
	\Z$, $\coh^1(F_\Sigma/F_\infty,M_s^*)$
	has no proper sub-$\Zp[[\Gamma]]$-modules of finite index.
\end{prop}
\begin{proof}
	Since $\coh^1(F_\Sigma/F_\infty,M_s^*)=\coh^1(F_\Sigma/F_\infty,M^*)\otimes \chi_{|\Gamma}^s$, it is enough to prove the proposition for $\coh^1(F_\Sigma/F_\infty,M^*)$.
	By Proposition~\ref{prop:surjection}, we have the short exact
	sequence
	
	\[
		0 \rightarrow \sel_{\underline{I}}(M^*/F_\infty)
		\rightarrow \coh^1(F_\Sigma/F_\infty,M^*) \rightarrow
		\PP_{\Sigma,\underline{I}}(M^*/F_\infty) \rightarrow
		0.
	\]
	We assume that $\sel_{\underline{I}}(M^*/F_\infty)$ is a cotorsion
	$\Zp[[\Gamma]]$-module.
	The above short exact sequence then forces the $\Zp[[\Gamma]]$-coranks of
	$\coh^1(F_\Sigma/F_\infty,M^*)$ and
	$\PP_{\Sigma,\underline{I}}(M^*/F_\infty)$ to be equal. Thus,
	by Lemma~\ref{lemm:corank}, we have
	\[
		\corank_{\Zp[[\Gamma]]} \coh^1(F_\Sigma/F_\infty,M^*)
		= g_+.
	\]
	On the other hand, from the global Euler-Poincar\'e characteristic formula
	\cite[Proposition 3]{GreenbergIwasawaRep}, we have
	\[
		\corank_{\Zp[[\Gamma]]} \coh^1(F_\Sigma/F_\infty,M^*)
		= \corank_{\Zp[[\Gamma]]}
		\coh^2(F_\Sigma/F_\infty,M^*) + \delta(F,\MM_p^*),
	\]
	with
	\[
		\delta(F,\MM_p^*) = \sum_{v \text{ complex}} \dim_{\Qp}
		\MM_p^* + \sum_{v \text{ real}} \dim_{\Qp}(\MM_p^*)^-
	\]
	where $v$ runs through archimedean primes of $F$ and, for a real
	prime $v$, and
	$\dim_{\Qp}(\MM_p^*)^-$ is the dimension of the $-1$-eigenspace
	for a complex conjugation above $v$ acting on $\MM_p$.
	From \cite[Eq. (34)]{GreenbergIwasawaRep}, we have
	\[
		\delta(F,\MM_p^*) = \dim_{\Qp}(\ind_F^\Q \MM_p^*)^- =
		g_+.
	\]
	Thus, $\coh^2(F_\Sigma/F_\infty,M^*)$ is a cotorsion
	$\Zp[[\Gamma]]$-module. 
	But by \cite[Proposition 4]{GreenbergIwasawaRep}, $\coh^2(F_\Sigma/F_\infty,M^*)$ is a cofree $\Zp[[\Gamma]]$-module, hence $\coh^2(F_\Sigma/F_\infty,M^*)=0$ and the proposition follows from \cite[Proposition 5]{GreenbergIwasawaRep}.
\end{proof}

\begin{rema}
	The weak Leopoldt conjecture~\cite[\S 1.3 and Appendix B]{PRLivre} would also imply that $\coh^1(F_\Sigma/F_\infty,M^*)$ has no proper sub-$\Zp[[\Gamma]]$-modules of finite index.
	Indeed, by the weak Leopoldt conjecture, $\coh^2(F_\Sigma/F_\infty,M^*)$ is trivial and we can apply \cite[Proposition 5]{GreenbergIwasawaRep}.
\end{rema}

\begin{proof}[Proof of Theorem~\ref{theo:submodules}]
	For any $s \in \Z$, since $\Gamma \simeq \Zp$ has
	$p$-cohomological dimension $1$, the restriction map
	\[
		\coh^1(F_\Sigma/F,M^*_s) \rightarrow
		\coh^1(F_\Sigma/F_\infty,M^*_s)^\Gamma
	\]
	is surjective. Thus, combined with Proposition~\ref{prop:surjection} and Lemma~\ref{lemm:surjection}, for all
	but finitely many $s \in \Z$, we obtain the commutative diagram
	\begin{center}
		\begin{tikzcd}
			\coh^1(F_\Sigma/F,M^*_s) \arrow{r} \arrow{d} &
			\coh^1(F_\Sigma/F_\infty,M^*_s)^\Gamma \arrow{r}
			\arrow{d} & 0 \\
			\PP_{\Sigma,\underline{I}}(M^*_s/F) \arrow{r}
			\arrow{d} &
			\PP_{\Sigma,\underline{I}}(M^*_s/F_\infty)^\Gamma
			\arrow{r} & 0 \\
			0 & & 
		\end{tikzcd}
	\end{center}
	which implies that 
	\[
		\coh^1(F_\Sigma/F_\infty,M^*_s)^\Gamma \rightarrow \PP_{\Sigma,\underline{I}}(M^*_s/F_\infty)^\Gamma
	\]
	is surjective.

	By Proposition~\ref{prop:surjection}, we have the short exact
	sequence
	\[
		0 \rightarrow \sel_{\underline{I}}(M^*_s/F_\infty)
		\rightarrow \coh^1(F_\Sigma/F_\infty,M^*_s) \rightarrow
		\PP_{\Sigma,\underline{I}}(M^*_s/F_\infty) \rightarrow
		0.
	\]
	Taking $\Gamma$-invariants gives the long exact sequence
	\[
		\coh^1(F_\Sigma/F_\infty,M^*_s)^\Gamma \rightarrow
		\PP_{\Sigma,\underline{I}}(M^*_s/F_\infty)^\Gamma
		\rightarrow \sel_{\underline{I}}(M^*_s/F_\infty)_\Gamma
		\rightarrow \coh^1(F_\Sigma/F_\infty,M^*_s)_\Gamma.
	\]
	Since the first map is surjective, $\sel_{\underline{I}}(M^*_s/F_\infty)_\Gamma
		\rightarrow \coh^1(F_\Sigma/F_\infty,M^*_s)_\Gamma$ is
		injective. Furthermore, $\coh^1(F_\Sigma/F_\infty,M^*_s)_\Gamma$
		is trivial by Proposition~\ref{prop:submodule}. 
	Thus, $\sel_{\underline{I}}(M^*_s/F_\infty)_\Gamma=0$ which implies the result.
\end{proof}

\subsection{An application: computation of the Euler-Poincar\'e characteristic}
We now choose a strongly admissible basis as in Lemma~\ref{lemm:BKSigned} and any
$\underline{I} \in \II_p$.

For $v$ a non-archimedean prime not dividing $p$, we recall the
definition of the Tamagawa number of $T$ at $v$ \cite[I \S
4]{FontainePerrinRiou}. 

If $N$ is $\Qp$-vector space of finite dimension $d$ (respectively a
free $\Zp$-module of rank $d$), we denote by $N^{-1}$ its linear dual and we
set $\det_{\Qp} N = \wedge_{\Qp}^d N$ (respectively $\det_{\Zp} N =
\wedge_{\Zp}^d N$). If $N$ is now a finitely generated $\Zp$-module, we
define the determinant of $N$ over $\Zp$ as
\[
	\det_{\Zp} N = (\det_{\Zp} N_{-1})^{-1} \otimes \det_{\Zp} N_0,
\]
where
\[
	0 \rightarrow N_{-1} \rightarrow N_0 \rightarrow N \rightarrow 0
\]
is a resolution of $N$ by free $\Zp$-modules of finite ranks $N_{-1}$
and $N_0$.  

We recall that we denote by $\MM_p$ the $p$-adic realization of our fixed motive $\MM$.
Let $\frob_v$ be the
geometric Frobenius in $\gal(F_{v,\mathrm{unr}}/F_v)$, we have an exact sequence
of $\Qp$-vector spaces
\[
	0 \rightarrow \coh^0(F_v,\MM_p) \rightarrow
	\coh^0(F_{v,\mathrm{unr}},\MM_p)
	\xrightarrow{1 - \frob_v} \coh^0(F_{v,\mathrm{unr}},\MM_p) \rightarrow
	\coh^1_{\mathrm{unr}}(F_v,\MM_p) \rightarrow 0
\]
which induces an isomorphism of $\Qp$-vector spaces
\[
	\iota_{\MM_p,v} : \det_{\Qp}\coh^0(F_v,\MM_p) \otimes
	(\det_{\Qp}\coh^1_{\mathrm{unr}}(F_v,\MM_p))^{-1} \rightarrow \Qp.
\]
Then 
\[
	\det_{\Zp}\coh^0(F_v,T) \otimes
	(\det_{\Zp}\coh^1_{\mathrm{unr}}(F_v,T))^{-1}
\]
is a $\Zp$-lattice inside $\det_{\Qp}\coh^0(F_v,\MM_p) \otimes
(\det_{\Qp}\coh^1_{\mathrm{unr}}(F_v,\MM_p))^{-1}$ and the Tamagawa number of $T$ at $v$,
denoted by $\tam_v(T)$, is defined as the unique power of $p$ such that
\[
	\iota_{\MM_p,v}( \det_{\Zp}\coh^0(F_v,T) \otimes
	(\det_{\Zp}\coh^1_{\mathrm{unr}}(F_v,T))^{-1} ) = \Zp \cdot \tam_v(T).
\]

We can now deduce the following corollary on the leading term of the
algebraic $p$-adic $L$-function, which is a generalization of Kim's
result on Kobayashi's plus/minus Selmer groups \cite[Theorem
1.2]{KimSubmodules}.
\begin{coro} \label{coro:EPchar}
	Assume that
	$\sel_{\mathrm{BK}}(M^*/F)$ is finite, and that
	$\sel_{\underline{I}^c}(M/F_\infty)$ is a cotorsion
	$\Zp[[\Gamma]]$-module. Denote by $(f_{\underline{I}}) \subset
	\Zp[[X]]$ the characteristic ideal of
	$\sel_{\underline{I}}(M^*/F_\infty)$.
	Then, up to a unit,
	\begin{equation} \label{eq:EPcharLemm}
		f_{\underline{I}}(0) = |\sel_{\mathrm{BK}}(M^*/F)| \cdot
		\prod_{v \nmid p} \tam_v(T).
	\end{equation}
\end{coro}
\begin{proof}
	First, we remark that the hypothesis implies that
	$\sel_{\underline{I}}(M^*/F_\infty)$ is a cotorsion
	$\Zp[[\Gamma]]$-module since, by Lemma~\ref{lemm:BKSigned},
	$\sel_{\mathrm{BK}}(M^*/F) = \sel_{\underline{I}}(M^*/F)$, and
	we have a ``control theorem'' (Lemma~\ref{lemm:control}).  Thus
	$\sel_{\underline{I}}(M^*/F_\infty)^\Gamma$ is finite, which
	implies that $\sel_{\underline{I}}(M^*/F_\infty)$ is
	cotorsion.

	Up to a unit, we have
	\begin{align} \label{eq:constantTerm}
		f_{\underline{I}}(0) 	& = |\sel_{\underline{I}}(M^*/F_\infty)^\Gamma|/|\sel_{\underline{I}}(M^*/F_\infty)_\Gamma| \\
			& = |\sel_{\underline{I}}(M^*/F_\infty)^\Gamma|,
			\nonumber
	\end{align}
	where the first relation is \cite[Lemma 4.2]{GreenbergIwasawaEll} and the
	second is Theorem~\ref{theo:submodules}.

	It remains to relate the right hand side of the formula
	\eqref{eq:EPcharLemm} to $|\sel_{\underline{I}}(M^*/F_\infty)^\Gamma|$. It
	is done by studying the commutative diagram
	\begin{center}
		\begin{tikzcd}
			0 \arrow{r} & \sel_{\mathrm{BK}}(M^*/F)
			\arrow{r} \arrow{d} & \coh^1(F_\Sigma/F,M^*)
			\arrow{r} \arrow{d} &
			\PP_{\Sigma,f}(M^*/F) \arrow{d}{(f_v)_v} \arrow{r} & 0 \\
			0 \arrow{r}  &
			\sel_{\underline{I}}(M^*/F_\infty)^{\Gamma}
			\arrow{r} &
			\coh^1(F_\Sigma/F_\infty,M^*)^{\Gamma} \arrow{r} &
			\PP_{\Sigma,\underline{I}}(M^*/F_\infty)^{\Gamma}
			\arrow{r} & 0,
		\end{tikzcd}
	\end{center}
	where the surjection at the end of the top row is
	Proposition~\ref{prop:surjection} and the one at the bottom row
	is due to Theorem~\ref{theo:submodules}.  As we mentioned in the
	proof of Lemma~\ref{lemm:control}, by \textbf{(Tors.)}, the
	central map is an isomorphism by the inflation-restriction exact
	sequence. Hence, by the snake lemma, we have
	\begin{equation} \label{eq:EPchar}
		|\sel_{\underline{I}}(M^*/F_\infty)^\Gamma| =
		|\sel_{\mathrm{BK}}(M^*/F)|.\prod_{v \in \Sigma} |\ker
		f_v|.
	\end{equation}
	We now compute $|\ker f_v|$. As we have already remarked, the
	archimedean part is trivial since $p$ is odd, and if $v$ divides
	$p$, then $f_v$ is injective (see the proof of
	Lemma~\ref{lemm:control}). Finally, if $v$ is a non-archimedean prime
	not dividing $p$, then $\ker f_v$ is the orthogonal complement
	under Tate's local pairing
	of the projection
	\[
		\im (\hiw(F_v,T) \xrightarrow{f_v^*} \hur(F_v,T)).
	\]
	By \cite[Lemme 2.2.5]{PRHauteurs}, we have $|\coker f_v^*|=\tam_v(T)$.
	Thus, $|\ker f_v| = \tam_v(T)$. Since $\tam_v(T) = 1$ at primes
	$v$ where $M^*$ is unramified and all the ramified primes of
	$M^*$ are contained in $\Sigma$, we can extend the product in
	\eqref{eq:EPchar} over
	all nonarchimedean primes not dividing $p$. The corollary
	follows from \eqref{eq:constantTerm} combined with
	\eqref{eq:EPchar}.
\end{proof}

\section{Congruences}
Let $\MM^\prime$ be another motive and $T^\prime$ a $G_F$-stable
$\Zp$-lattice inside $\MM_p^\prime$ the $p$-adic realization of
$\MM^\prime$ satisfying all the hypotheses \textbf{(H.-T.)},
\textbf{(Cryst.)}, \textbf{(Tors.)}, \textbf{(Fil.)} and
\textbf{(Slopes)}.  We shall simply add a superscript $(\cdot)^\prime$
to the various object associated to $T$ to denote the similar object
associated to $T^\prime$ (\textit{e.g.} $M^\prime = T^\prime \otimes
\Qp/\Zp$).

From now on, we assume that
\begin{equation}
	T/pT \simeq T^\prime/pT^\prime,\tag*{\textbf{(Cong.)}}
\end{equation}
as $G_F$-representations.
The goal of this section is to compare
the signed Selmer groups of $M^*$ and $M^{\prime,*}$ under the hypothesis
\textbf{(Cong.)}. 
We begin by studying the implication of this congruence on the signed
Coleman maps.

\subsection{Wach modules}
We succinctly recall what we need for our purpose and refer the reader
to \cite{BergerExp,BergerLimites} for details. Let $v$ be a prime of $F$
dividing $p$ and let $\A$ be the ring $\OO_{F_v}[[\pi]]$ equipped with
the semilinear action by the Frobenius $\phi$ which acts as the absolute
Frobenius on $\OO_{F_v}$ and on $\pi$ by
\[
	\phi(\pi) = (\pi + 1)^p - 1,
\]
and with an action of $\gal(F_v(\mu_{p^\infty})/F_v)$ given by 
\[
	g(\pi) = (\pi + 1)^{\chi(g)} - 1, \forall g \in
	\gal(F_v(\mu_{p^\infty})/F_v).
\]
There exists a Wach module $\N_v(T)$ attached to $T|_{G_{F_v}}$ (and
similarly to $T'|_{G_{F_v}}$), which is a free $\A$-module of rank
$\dim_{\Qp}(\MM_p)$ equipped with an action of
$\gal(F_v(\mu_{p^\infty})/F_v)$ and a $\phi$-linear endomorphism of
$\N_v(T)[\frac{1}{\pi}]$, which we still denote by $\phi$, commuting
with the Galois action. We denote by $\phi^* \N_v(T)$ the $\A$-module
generated by $\phi(\N_v(T))$.

The Dieudonn\'e module associated to $T$ is defined via the Wach module
by
\begin{equation} \label{eq:WachDieudonne}
	\N_v(T)/\pi\N_v(T) = \Dcrisv(T),
\end{equation}
where the filtration on $\N_v(T)$ inducing the one on $\Dcrisv(T)$ is 
\[
	\fil^i \N_v(T) = \{x \in \N_v(T), \phi(x) \in
	(\phi(\pi)/\pi)^i \N_v(T)\},
\]
One also recovers the first Iwasawa cohomology group from the Wach
module by an isomorphism of
$\Zp[[\gal(F_v(\mu_{p^\infty})/F_v)]]$-modules
\begin{equation} \label{eq:WachIwasawa}
	h_T^1: \hiw(F_v,T) \xrightarrow{\sim} \N_v(T)^{\psi=1},
\end{equation}
where $\psi$ is a left inverse for $\phi$. 

We compare the Wach module of $T$ and $T^\prime$ modulo
$p$.
Since the Hodge-Tate weights of $T$ and $T^\prime$ are in $[0,1]$, the
following theorem is a special case of \cite[Th\'eor\`eme
IV.1.1]{BergerLimites}.
\begin{theo} \label{theo:Wachmodp}
	The isomorphism \textbf{(Cong.)} induces an isomorphism of $\A$-modules
	\[
		\N_v(T)/p\N_v(T) \simeq \N_v(T^\prime)/p\N_v(T^\prime),
	\]
	which is compatible with the filtration, the Galois action and the action of $\phi$.
\end{theo}

\subsection{Congruences of signed Coleman maps}
\label{subsec:congruencesColeman}
We now follow the construction of the signed Coleman maps as given in
\cite[\S 2]{BLIntegral} keeping track of the congruences modulo $p$.

First, note that by \eqref{eq:WachIwasawa} and
Theorem~\ref{theo:Wachmodp}, we have a
$\Zp[[\gal(F_v(\mu_{p^\infty})/F_v)]]$-isomorphism
\begin{equation}\label{eq:isohiw} \hiw(F_v,T)/p \simeq
\hiw(F_v,T^\prime)/p.  \end{equation} Also, combining
\eqref{eq:WachDieudonne} and Theorem~\ref{theo:Wachmodp}, the
Dieudonn\'e modules associated to $T$ and $T^\prime$ are isomorphic
modulo $p$. We fix good bases for $\Dcrisv(T)$ and $\Dcrisv(T^\prime)$
compatible with the isomorphism given in Theorem~\ref{theo:Wachmodp} in
the sense that they have the same image under \eqref{eq:WachDieudonne}.

\begin{lemm} \label{lemm:modpLn}
	For $n \geqslant 1$, there exists a unique
	$\Lambda$-homomorphism
	\[
		\LL^{(n)}_T : \hiw(F_v,T) \rightarrow \Lambda_n \otimes_{\Zp} \Dcrisv(T)
	\]
	such that 
	\[
		\phi^{-n-1}\circ \LL_T \equiv \LL^{(n)}_T \mod
		\omega_n.
	\]
	Furthermore, the applications $\LL^{(n)}_T$ and
	$\LL^{(n)}_{T^\prime}$ are congruent modulo $p$, \textit{i.e.} the diagram
	\begin{center}
	\begin{tikzcd}
		\hiw(F_v,T)/p \arrow{r}{\LL^{(n)}_T \mod p}
		\arrow{d}{\simeq} & \Lambda_n
		\otimes_{\Zp} \Dcrisv(T)/p \arrow{d}{\simeq}\\
		\hiw(F_v,T^\prime)/p \arrow{r}{\LL^{(n)}_{T^\prime} \mod p} & \Lambda_n
		\otimes_{\Zp} \Dcrisv(T^\prime)/p 
	\end{tikzcd}
	\end{center}
	is commutative.
\end{lemm}
\begin{proof}
	The first statement is Proposition 2.9 of \textit{op. cit.}. We
	follow its proof to prove the second.	
	Perrin-Riou's big logarithm $\LL_T$ is given by 
	\[
		(\mathfrak{M}\otimes 1)^{-1}\circ(1-\phi)\circ(h_T)^{-1}
	\]
	where $\mathfrak{M}$ is the Mellin transform which maps elements
	of $\HH$ to overconvergent power series in $\pi$, whose set we
	denote $\Brig$. The Mellin
	transform preserves integrality and the ideal $(\omega_n)$
	corresponds to $(\phi^{n+1}(\pi))$.

	The first statement then follows from a study (Lemma~3.44 of
	\textit{op. cit.}) of the map
	\begin{equation} \label{eq:phin}
		\phi^{-n-1}\circ (1-\phi) : \N_v(T)^{\psi=1} \rightarrow
		(\phi^*\N_v(T))^{\psi=0} \hookrightarrow
		\Brig\otimes_{\Zp}\Dcrisv(T)
	\end{equation}
	which shows that, for $x \in \N_v(T)^{\psi=1}$, the element
	$\phi^{-n-1}\circ (1-\phi)(x)$ is congruent to an element of
	$(\A)^{\psi=0}\otimes_{\Zp} \Dcrisv(T)$ modulo
	$\phi^{n+1}(\pi)\Brig\otimes_{\Zp}\Dcrisv(T)$. But by
	Theorem~\ref{theo:Wachmodp}, the maps \eqref{eq:phin} for $T$ and $T^\prime$
	agree modulo $p$ and we are done. 
\end{proof}

For $i \in \{1,\dots,g_v\}$, we write $\LL_{T,i}^{(n)}$ for the
composition of $\LL_T^{(n)}$ with the projection on the $i$-th component
of the fixed basis of $\Dcrisv(T)$.  We set $h_n$ (respectively
$h_n^\prime$) the
$\Lambda_n$-endomorphisms on $\oplus_{k=1}^{g_v}\Lambda_n$ given by the
left multiplication by the product $C_{v,n} \cdots C_{v,1}$ (respectively
$C_{v,n}^\prime \cdots C_{v,1}^\prime$).
\begin{lemm} \label{lemm:modpColn}
	For $n \geqslant 1$, there exists a unique
	$\Lambda$-homomorphism
	\[
		\col^{(n)}_T : \hiw(F_v,T) \rightarrow
		\bigoplus_{k=1}^{g_v} \Lambda_n
	\]
	such that 
	\[
		\begin{pmatrix}
			\LL^{(n)}_{T,1}\\
			\vdots\\
			\LL^{(n)}_{T,g_v}
		\end{pmatrix}
		 \equiv C_{v,n} \cdots C_{v,1} \cdot \col^{(n)}_T  \mod
		\ker h_n.
	\]
	Furthermore, we have
	\[
		\col^{(n)}_T \equiv \col^{(n)}_{T^\prime} \mod p.
	\]
\end{lemm}
\begin{proof}
	The first part is Proposition~2.10 of \textit{op. cit.}.
	Again by Theorem~\ref{theo:Wachmodp} and \eqref{eq:WachDieudonne},
	the matrices $C_{v,n}$ and $C_{v,n}^\prime$ are congruent modulo
	$p$ for all $n$.
	Thus, by the first part of the Lemma and Lemma~\ref{lemm:modpLn}, we have
	\[
		C_{v,n} \cdots C_{v,1} \cdot \col^{(n)}_T \equiv C_{v,n}
		\cdots C_{v,1} \cdot \col^{(n)}_{T^\prime} \mod (\ker
		h_n,p).  
	\]
	Since 
	\[
		C_{v,n} = \left( \begin{array}{c|c}
			I_{d_v} & 0 \\
			\hline
			0 & \Phi_{p^n}(1+X) I_{g_v - d_v}
		\end{array}\right) C_v^{-1},
	\]
	with $C_v \in \gl_{g_v}(\Zp)$ and $\Phi_{p^n}(1+X)$ and $p$ are
	coprime, the second part follows.
\end{proof}

By \cite[Lemma 2.11 and Theorem 2.13]{BLIntegral}, the maps
$(\col_T^{(n)})_{n \geqslant 1}$ are compatible with the natural
projection $\oplus_{k=1}^{g_v}\Lambda_{n+1}/\ker h_{n+1} \rightarrow
\oplus_{k=1}^{g_v}\Lambda_n/\ker h_n$ and thus define a map to
$\varprojlim_n \oplus_{k=1}^{g_v}\Lambda_n/\ker h_n$ which naturally
identifies with $\oplus_{k=1}^{g_v} \Lambda$.  By definition, this map
is $\col_T$. Thus, by Lemma~\ref{lemm:modpColn} we have :

\begin{prop} \label{prop:modplocalp}
	The Coleman maps associated to $T$ and $T^\prime$ are congruent
	modulo $p$. More precisely, if $z\in\hiw(F_v,T) $ and
	$z'\in\hiw(F_v,T') $ have the same image under the isomorphism
	given in \eqref{eq:isohiw}, then
	\[
	\col_T(z)\equiv \col_{T'}(z')\mod p\oplus_{k=1}^{g_v}\Lambda.
	\]
\end{prop}

\subsection{Non-primitive signed Selmer groups}
We now define and compare the \emph{non-primitive signed} Selmer groups under the
hypothesis \textbf{(Cong.)} and deduce our main result.

The next lemma is well-known~\cite[Lemma 3.5.3]{MazurRubin04}.
\begin{lemm} \label{lemm:modpcoh}
	The exact sequence 
	\[ 
		0 \rightarrow M^*[p] \rightarrow M^* \xrightarrow{p} M^*
		\rightarrow 0
	\]
	of $G_{F_\infty}$-modules
	induces isomorphisms
	\[
		\coh^1(F_\Sigma/F_\infty,M^*[p])\simeq
		\coh^1(F_\Sigma/F_\infty,M^*)[p],
	\]
	and 
	\[
		\coh^1(F_{v,\infty},M^*[p])\simeq
		\coh^1(F_{v,\infty},M^*)[p],
	\]
	for any prime $v$ dividing $p$,
	and, for any non-archimedean prime $v$ not dividing $p$ or a
	prime of ramification of $M^*$,
	\[
		\coh^1(F_{v,\mathrm{unr}},M^*[p])\simeq
		\coh^1(F_{v,\mathrm{unr}},M^*)[p].
	\]
\end{lemm}
\begin{proof}
	We have the exact sequence
	\[
		0 \rightarrow  \coh^0(F_\infty,M^*)/p \rightarrow
		\coh^1(F_\Sigma/F_\infty,M^*[p]) \rightarrow
		\coh^1(F_\Sigma/F_\infty,M^*)[p]\rightarrow 0.
	\]
	Since $\coh^0(F_\infty,M^*)$ is trivial by our hypothesis
	\textbf{(Tors.)}, we get the first isomorphism. The same proof
	applies for the second isomorphism at $v$ dividing $p$.
	Let $v$ as in the third statement, then $v$ is not a prime of
	ramification for $M^*$, thus $\coh^0(F_{v,\mathrm{unr}},M^*) =
	M^*$. Hence, we have the exact sequence 
	\[
		0 \rightarrow M^*/p \rightarrow 
		\coh^1(F_{v,\mathrm{unr}},M^*[p])\rightarrow
		\coh^1(F_{v,\mathrm{unr}},M^*)[p] \rightarrow 0.
	\]
	Since $M^*$ is divisible, we deduce the second isomorphism.
\end{proof}

\begin{defi} \label{defi:non-primitive}
	Let $\Sigma_0 \subset \Sigma$ be a subset that contains all the
	primes of ramification of $M^*$ but none of the primes of $F$ dividing $p$
	nor the archimedean primes.
	We define the \emph{$\Sigma_0$-non-primitive} $\underline{I}$-Selmer groups
	of $M^*$ over $F_\infty$
	by
	\[
		\sel_{\underline{I}}^{\Sigma_0}(M^*/F_\infty) =
		\ker(\coh^1(F_\Sigma/F_\infty,M^*) \rightarrow
		\PP_{\Sigma \setminus
		\Sigma_0,\underline{I}}(M^*/F_\infty)).
	\]
	If $v$ is a prime dividing $p$, by Lemma~\ref{lemm:modpcoh} we
	have $\coh^1(F_{v,\infty},M^*[p])\simeq
	\coh^1(F_{v,\infty},M^*)[p]$ and we set 
	\[
		\coh^1_{I_v}(F_{v,\infty},M^*[p]) =
		\coh^1_{I_v}(F_{v,\infty},M^*)[p] \subset
		\coh^1(F_{v,\infty},M^*[p]).
	\]
	We set 
	\[
		\PP_{\Sigma \setminus
		\Sigma_0,\underline{I}}(M^*[p]/F_\infty) = \prod_{w \in
		\Sigma \setminus \Sigma_0, w \nmid p}
		\coh^1(F_{w,\mathrm{unr}},M^*[p])
		\times \prod_{w \mid p}
		\frac{\coh^1(F_{w,\infty},M^*[p])}{\coh^1_{I_v}(F_{w,\infty},M^*[p])}.
	\]
	We define the \emph{$\Sigma_0$-non-primitive} $\underline{I}$-Selmer
	groups of $M^*[p]$ over $F_\infty$ by
	\[
		\sel_{\underline{I}}^{\Sigma_0}(M^*[p]/F_\infty) =
		\ker(\coh^1(F_\Sigma/F_\infty,M^*[p]) \rightarrow
		\PP_{\Sigma \setminus
		\Sigma_0,\underline{I}}(M^*[p]/F_\infty)).
	\]
\end{defi}

\begin{rema} \label{rema:modp}
	By Tate's pairing \eqref{eq:cohompairing}, the Pontryagin dual of
	$\coh^1_{I_v}(F_{v,\infty},M^*[p])$ is $(\im
	\col_{T,I_v})/p$.
\end{rema}

From now on, we write abusively the $\mu$ and $\lambda$-invariants of
the various Selmer groups to refer to the $\mu$ and $\lambda$-invariants
of their Pontryagin duals.

\begin{prop} \label{prop:modp}
	For any $\Sigma_0 \subset \Sigma$ as in
	Definition~\ref{defi:non-primitive}, we have an isomorphism of
	$\Zp[[\Gamma]]$-modules
	\[
		\sel_{\underline{I}}^{\Sigma_0}(M^*[p]/F_\infty) \simeq 
		\sel_{\underline{I}}^{\Sigma_0}(M^*/F_\infty)[p].
	\]
\end{prop}
\begin{proof}
	By Lemma~\ref{lemm:modpcoh}, we have
	\[
		\coh^1(F_\Sigma/F_\infty,M^*[p])\simeq
		\coh^1(F_\Sigma/F_\infty,M^*)[p].
	\]
	Therefore, in order to prove the Proposition, it is enough to
	compare the local conditions defining the two Selmer groups.
	At $v \in \Sigma \setminus \Sigma_0$ and $v$ not dividing $p$, the
	second part of Lemma~\ref{lemm:modpcoh} shows that the local
	conditions are equivalent. Since $p$ is odd, the
	archimedean part is trivial. At $v$ dividing $p$, by definition
	the local conditions are the same.
\end{proof}

We now relate the $\Sigma_0$-non-primitive signed Selmer groups to the signed Selmer groups.
\begin{prop}\label{prop:relationSelmer}
Assume that $\sel_{\underline{I}^c}(M/F_\infty)$ is a cotorsion $\Zp[[\Gamma]]$-module. 
Then 
	\[
			\sel_{\underline{I}}^{\Sigma_0}(M^*/F_\infty)/\sel_{\underline{I}}(M^*/F_\infty)
			\simeq \prod_{w \in \Sigma_0}
			\frac{\coh^1(F_{\infty,w},M^*)}{\hur(F_{\infty,w},M^*)}.
	\]
\end{prop}
\begin{proof}
By Proposition~\ref{prop:surjection}, we have the commutative diagram
\begin{equation} \label{diag:compare}
	\begin{tikzcd}
		0 \arrow{r} &
		\sel_{\underline{I}}(M^*/F_\infty)
		\arrow{r} \arrow{d} & \coh^1(F_\Sigma/F_\infty,M^*) \arrow{r} \arrow{d}{=} &
		\PP_{\Sigma,\underline{I}}(M^*/F_\infty)
		\arrow{r} \arrow{d}{h_{\Sigma_0}}  & 0\\
		0 \arrow{r} &
		\sel_{\underline{I}}^{\Sigma_0}(M^*/F_\infty)
		\arrow{r} & \coh^1(F_\Sigma/F_\infty,M^*) \arrow{r} &
		\PP_{\Sigma\setminus
		\Sigma_0,\underline{I}}(M^*/F_\infty).
		& 
	\end{tikzcd}
\end{equation}
The Snake lemma applied to diagram~\eqref{diag:compare} gives
\[
	\sel_{\underline{I}}^{\Sigma_0}(M^*/F_\infty)/\sel_{\underline{I}}(M^*/F_\infty)
	\simeq \ker(h_{\Sigma_0})=\prod_{w \in \Sigma_0}
	\frac{\coh^1(F_{\infty,w},M^*)}{\hur(F_{\infty,w},M^*)}.
\]
\end{proof}

\begin{coro} \label{coro:relationSelmer}
Assume that $\sel_{\underline{I}^c}(M/F_\infty)$ and $\sel_{\underline{I}}(M^*/F_\infty)$ are cotorsion $\Zp[[\Gamma]]$-modules.
	Then $\sel_{\underline{I}}^{\Sigma_0}(M^*/F_\infty)$ is a cotorsion $\Zp[[\Gamma]]$-module. Furthermore, we have 
	\[
		\mu(\sel_{\underline{I}}^{\Sigma_0}(M^*/F_\infty))=\mu(\sel_{\underline{I}}(M^*/F_\infty)),
	\]
	and
	\[
	\corank_{\Zp} \sel_{\underline{I}}^{\Sigma_0}(M^*/F_\infty) =
	\corank_{\Zp} \sel_{\underline{I}}(M^*/F_\infty) + \sum_{w\in
	\Sigma_0} \corank_{\Zp} \coh^1(F_{\infty,w},M^*).
	\]
\end{coro}
\begin{proof}
We have already noted that $\hur(F_{\infty,w},M^*)$ is
trivial (by \cite[\S A.2.4]{PRLivre}) and that $\coh^1(F_{\infty,w},M^*)$ is a cotorsion
$\Zp[[\Gamma]]$-module for all $w \in \Sigma_0$ (by \cite[Proposition 2]{GreenbergIwasawaRep}). 
Furthermore, the $\mu$-invariant of the Pontryagin dual of $\coh^1(F_{\infty,w},M^*)$ is zero
\cite[Proposition 2.4]{GreenbergVatsal}.
	Thus the corollary follows from Proposition~\ref{prop:relationSelmer}.
\end{proof}

We have an analogue of Theorem~\ref{theo:submodules} for $\Sigma_0$-non-primitive
signed Selmer groups.
\begin{prop} \label{prop:submodulesNonPrimitive}
	Assume that
	$\sel_{\underline{I}}(M^*/F_\infty)$ and
	$\sel_{\underline{I}^c}(M/F_\infty)$ are cotorsion
	$\Zp[[\Gamma]]$-modules. Then
	$\sel_{\underline{I}}^{\Sigma_0}(M^*/F_\infty)$ has no proper
	sub-$\Zp[[\Gamma]]$-modules of finite index.
\end{prop}
\begin{proof}
	From the definition of the $\Sigma_0$-non-primitive signed Selmer groups
	and Proposition~\ref{prop:surjection}, we have an analogue of
	Proposition~\ref{prop:surjection} for the $\Sigma_0$-non-primitive signed
	Selmer groups. That is, in diagram~\eqref{diag:compare}, the surjectivity of $h_{\Sigma_0}$ implies that 
	the sequence
	\begin{equation} \label{diag:analoguesurjection}
		0 \rightarrow
		\sel_{\underline{I}}^{\Sigma_0}(M^*/F_\infty)
		\rightarrow \coh^1(F_\Sigma/F_\infty,M^*) \rightarrow
		\PP_{\Sigma\setminus
		\Sigma_0,\underline{I}}(M^*/F_\infty)
		\rightarrow 0
	\end{equation}
	is exact.
	Similarly, by Proposition~\ref{prop:surjection}, we have
	\[
		0 \rightarrow
		\sel_{\underline{I}}^{\Sigma_0}(M^*/F)
		\rightarrow \coh^1(F_\Sigma/F,M^*) \rightarrow
		\PP_{\Sigma\setminus \Sigma_0,\underline{I}}(M^*/F)
		\rightarrow 0.
	\]
	The proof of the Proposition then follows precisely the one of
	Theorem~\ref{theo:submodules}, hence, we skip it.
\end{proof}

\begin{coro} \label{coro:lambda}
	Assume that
	$\sel_{\underline{I}}(M^*/F_\infty)$ and
	$\sel_{\underline{I}^c}(M/F_\infty)$ are cotorsion
	$\Zp[[\Gamma]]$-modules. Furthemore, assume that the
	$\mu$-invariant of $\sel_{\underline{I}}(M^*/F_\infty)$ is zero.
	Then the $\lambda$-invariant of
	$\sel_{\underline{I}}^{\Sigma_0}(M^*/F_\infty)$ is equal to $\dim_{\Fp}
	\sel_{\underline{I}}^{\Sigma_0}(M^*[p]/F_\infty)$.
\end{coro}
\begin{proof} 
	By Corollary~\ref{coro:relationSelmer}, $\sel_{\underline{I}}^{\Sigma_0}(M^*/F_\infty)$ is a cotorsion
	$\Zp[[\Gamma]]$-module and its
	$\mu$-invariant is zero. 
	Thus, the Pontryagin dual of
	$\sel_{\underline{I}}^{\Sigma_0}(M^*/F_\infty)$ is a finitely
	generated $\Zp$-module, and, by
	Proposition~\ref{prop:submodulesNonPrimitive}, its $\Zp$-torsion
	submodule is trivial. Thus,
	$\sel_{\underline{I}}^{\Sigma_0}(M^*/F_\infty)$ is
	$\Zp$-divisible with $\Zp$-corank its $\lambda$-invariant (\textit{i.e.}
	$\sel_{\underline{I}}^{\Sigma_0}(M^*/F_\infty) \simeq
	(\Qp/\Zp)^\lambda$). Therefore, by Proposition~\ref{prop:modp},
	we have 
	\[
		\lambda =
		\dim_{\Fp}\sel_{\underline{I}}^{\Sigma_0}(M^*/F_\infty)[p]
		=
		\dim_{\Fp}\sel_{\underline{I}}^{\Sigma_0}(M^*[p]/F_\infty).
	\]
\end{proof}

We are now able to prove the main result of this section.
\begin{theo} \label{theo:congruences}
	Assume \textbf{(Cong.)} and choose compatible good bases for $T$
	and $T'$ as in \S \ref{subsec:congruencesColeman}. Let
	$\underline{I} \in \II_p$ and $\Sigma_0$ a finite set of prime
	as in Definition~\ref{defi:non-primitive} containing the primes
	of ramification of $M^*$ and $M^{\prime,*}$. Further, assume that
	$\sel_{\underline{I}}(M^*/F_\infty)$,
	$\sel_{\underline{I}^c}(M/F_\infty)$,
	$\sel_{\underline{I}}(M^{\prime,*}/F_\infty)$ and
	$\sel_{\underline{I}^c}(M^\prime/F_\infty)$ are cotorsion
	$\Zp[[\Gamma]]$-modules.  Then the $\mu$-invariant of
	$\sel_{\underline{I}}(M^*/F_\infty)$ vanishes if and only if
	that of $\sel_{\underline{I}}(M^{\prime,*}/F_\infty)$ vanishes.
	Furthermore, when these $\mu$-invariants do vanish, the
	$\lambda$-invariants of
	$\sel_{\underline{I}}^{\Sigma_0}(M^*/F_\infty)$ and
	$\sel_{\underline{I}}^{\Sigma_0}(M^{\prime,*}/F_\infty)$ are
	equal.
\end{theo}
\begin{proof}
	The hypothesis \textbf{(Cong.)} implies that $M^*[p] \simeq
	M^{\prime,*}[p]$ as $G_F$-representations. 
	Hence, we have
	\begin{align*}
		\coh^1(F_\Sigma/F_\infty,M^*[p]) & \simeq
		\coh^1(F_\Sigma/F_\infty,M^{\prime,*}[p]), \\
		\coh^1(F_{v,\mathrm{unr}},M^*[p]) & \simeq
		\coh^1(F_{v,\mathrm{unr}},M^{\prime,*}[p]), \\
	\end{align*}
	for $v$ not dividing $p$.
	By Remark~\ref{rema:modp} and Proposition~\ref{prop:modplocalp}, for each $v$ dividing $p$, we
	have
	\[
		\coh^1_{I_v}(F_{v,\infty},M^*[p]) \simeq
		\coh^1_{I_v}(F_{v,\infty},M^{\prime,*}[p]).
	\]
	It follows that
	\[
		\sel_{\underline{I}}^{\Sigma_0}(M^*[p]/F_\infty) \simeq 
		\sel_{\underline{I}}^{\Sigma_0}(M^{\prime,*}[p]/F_\infty).
	\]
	Combined with 
	Proposition~\ref{prop:modp}, we have
	\[
		\sel_{\underline{I}}^{\Sigma_0}(M^*/F_\infty)[p] \simeq
		\sel_{\underline{I}}^{\Sigma_0}(M^*[p]/F_\infty) \simeq 
		\sel_{\underline{I}}^{\Sigma_0}(M^{\prime,*}[p]/F_\infty) \simeq
		\sel_{\underline{I}}^{\Sigma_0}(M^{\prime,*}/F_\infty)[p].
	\]
	Therefore, the first assertion follows from Corollary~\ref{coro:relationSelmer}.
	The second claim follows from Corollary~\ref{coro:lambda}.
\end{proof}

As a direct consequence of Theorem~\ref{theo:congruences} and Corollary~\ref{coro:relationSelmer}, we get:
\begin{coro}\label{coro:congruences}
	Assume the same hypotheses as in Theorem~\ref{theo:congruences} and also assume that 
	\[
		\mu(\sel_{\underline{I}}(M^*/F_\infty))= \mu(\sel_{\underline{I}}(M^{\prime,*}/F_\infty))=0.
	\]
	Then 
	\[
		\lambda - \sum_{w \in \Sigma_0}
		\corank_{\Zp} \coh^1(F_{\infty,w},M^*) = \lambda^\prime - \sum_{w \in \Sigma_0}
		\corank_{\Zp} \coh^1(F_{\infty,w},M^{\prime,*}),
	\]
	where $\lambda$ (respectively $\lambda^\prime$) denotes the $\lambda$-invariant of $\sel_{\underline{I}}(M^*/F_\infty)$ (respectively $\sel_{\underline{I}}(M^{\prime,*}/F_\infty)$).
\end{coro}
\begin{rema}
	In Corollary~\ref{coro:congruences}, we can compute the $\Zp$-coranks of $\coh^1(F_{\infty,w},M^*)$ and
$\coh^1(F_{\infty,w},M^{\prime,*})$ thanks to \cite[Proposition
2.4]{GreenbergVatsal}. Let $v$ be the prime of $F$ under $w$ and $\ell$
be the rational prime (different from $p$) under $v$.
We denote by $(\MM_p^*)_{F_{v,\mathrm{unr}}}$ the maximal quotient of
$\MM_p^*$ on which the group $\gal(\overline{F_v}/F_{v,\mathrm{unr}})$
acts trivially and we set $P_v(X)=\det(1-\frob_v X |
(\MM_p^*)_{F_{v,\mathrm{unr}}}) \in \Zp[X]$. Then the corank of
$\coh^1(F_{\infty,w},M^*)$ is equal to the multiplicicy of
$\ell^{-[F_v:\Qp]}$ as root of $P_v(X)$ in $\Fp[X]$.
\end{rema}

\bibliographystyle{amsalpha}
\bibliography{references.bib}
\end{document}